\documentclass[12pt,a4paper,leqno]{amsart}
\usepackage[french,english]{babel}

\usepackage[latin1]{inputenc}
\usepackage[T1]{fontenc}
\usepackage{amsfonts}
\usepackage{amsmath}
\usepackage{amssymb}
\usepackage{eurosym}
\usepackage{mathrsfs}
\usepackage{palatino}
\usepackage{color}
\usepackage{esint}
\usepackage{url}
\usepackage{verbatim}

\usepackage{enumerate}

\newcommand{\R}{\mathbb{R}}
\newcommand{\C}{\mathbb{C}}

\newcommand{\Z}{\mathbb{Z}}
\newcommand{\E}{\mathbb{E}}

\newcommand{\calR}{\mathcal{R}}

\newcommand{\bla}{\big \langle}
\newcommand{\bra}{\big \rangle}

\numberwithin{equation}{section}

\newcommand{\ud}[0]{\,\mathrm{d}}

\newcommand{\esssup}[0]{\operatornamewithlimits{ess\,sup}}

\newcommand{\BMO}[0]{\operatorname{BMO}}
\newcommand{\bmo}[0]{\operatorname{bmo}}

\newcommand{\loc}[0]{\operatorname{loc}}

\newcommand{\calD}[0]{\mathcal{D}}

\newcommand{\wt}[1]{{\widetilde{#1}}}

\swapnumbers
\theoremstyle{plain}
\newtheorem{thm}[equation]{Theorem}
\newtheorem{lem}[equation]{Lemma}

\theoremstyle{definition}

\theoremstyle{remark}
\newtheorem{rem}[equation]{Remark}

\author{Kangwei Li}
\address[K.L.]{BCAM (Basque Center for Applied Mathematics), Alameda de Mazarredo 14, 48009 Bilbao, Spain} 
\email{kli@bcamath.org}

\author{Henri Martikainen}
\address[H.M.]{Department of Mathematics and Statistics, University of Helsinki, P.O.B. 68, FI-00014 University of Helsinki, Finland}
\email{henri.martikainen@helsinki.fi}

\author{Emil Vuorinen}
\address[E.V.]{Department of Mathematics and Statistics, University of Helsinki, P.O.B. 68, FI-00014 University of Helsinki, Finland}
\email{emil.vuorinen@helsinki.fi}

\title[Bilinear bi-parameter commutators]{Commutators of bilinear bi-parameter singular integrals}

\makeatletter
\@namedef{subjclassname@2010}{%
  \textup{2010} Mathematics Subject Classification}
\makeatother


\begin{document}

\begin{abstract}
We study the boundedness properties of commutators formed by $b$ and $T$, where
$T$ is a bilinear bi-parameter singular integral satisfying natural $T1$ type conditions and $b$ is a little BMO function.
For paraproduct free bilinear bi-parameter singular integrals $T$ we prove that
$[b, T]_1 \colon L^p(\mathbb{R}^{n+m}) \times L^q(\mathbb{R}^{n+m}) \to L^r(\mathbb{R}^{n+m})$ in the full range
$1 < p, q \le \infty$, $1/2 < r < \infty$ satisfying $1/p+1/q = 1/r$. A special case is when $T$ is a bilinear bi-parameter multiplier.
We also prove the corresponding Banach range result
for all singular integrals satisfying the $T1$ type conditions. In doing so we simplify the corresponding linear proof. 
Lastly, we prove analogous results for iterated commutators.
\end{abstract}

\maketitle

\section{Introduction}
This paper concerns commutator estimates for general bilinear bi-parameter singular integrals $T$. Examples of such operators
include the bilinear bi-parameter multiplier operators $T_m$ studied in Muscalu--Pipher--Tao--Thiele \cite{MPTT}:
$$
T_m(f_1, f_2)(x) = \iint_{\R^{n+m}} \iint_{\R^{n+m}} m(\xi, \eta) \widehat f_1(\xi) \widehat f_2(\eta) e^{2\pi i x \cdot (\xi + \eta)} \ud \xi \ud \eta,
$$
where
$$
|\partial^{\alpha_1}_{\xi_1} \partial^{\alpha_2}_{\xi_2}  \partial^{\beta_1}_{\eta_1}  \partial^{\beta_2}_{\eta_2}  m(\xi, \eta)| \lesssim
(|\xi_1| + |\eta_1|)^{-|\alpha_1| - |\beta_1|} (|\xi_2| + |\eta_2|)^{-|\alpha_2| - |\beta_2|}. 
$$
See Coifman--Meyer \cite{CM} and Grafakos--Torres \cite{GT} for the one-parameter theory of such multipliers.
A general definition of a (not necessarily of tensor
product or convolution type) bilinear bi-parameter singular integral was given in our previous paper \cite{LMV}.
There we showed a dyadic representation theorem under $T1$ type assumptions, and used it to conclude various boundedness
properties, including weighted estimates $L^p(w_1) \times L^q(w_2) \to L^r(v_3)$, where
$1 < p, q < \infty$, $1/2 < r < \infty$, $1/p+1/q = 1/r$,
$w_1 \in A_p(\R^n \times \R^m)$, $w_2 \in A_q(\R^n \times \R^m)$ and $v_3 := w_1^{r/p} w_2^{r/q}$.
Here we complement
these results and provide further use for our recent bilinear bi-parameter representation theorem by proving commutator estimates. A very special case of our results
implies that
$$
\|[b,T_m]_1(f_1, f_2)\|_{L^r(\R^{n+m})} \lesssim \|b\|_{\bmo(\R^{n+m})} \|f_1\|_{L^p(\R^{n+m})} \|f_2\|_{L^q(\R^{n+m})} 
$$
for all $1 < p, q \le \infty$ and $1/2 < r < \infty$ satisfying $1/p+1/q = 1/r$, where $b$ is in little BMO, $T_m$ is a
bi-parameter multiplier and $[b,T_m]_1(f_1,f_2) := bT_m(f_1, f_2) - T_m(bf_1, f_2)$. Our main theorem for the first order commutator is:
\begin{thm}
Let $1 < p,q \le \infty$ and $1/2 < r < \infty$ satisfy $1/p + 1/q = 1/r$, and let $b \in \bmo(\R^{n+m})$.
Suppose $T$ is a bilinear bi-parameter Calder\'on--Zygmund operator satisfying all the structural assumptions and all the
boundedness and cancellation assumptions as formulated in Section 3 of \cite{LMV}. Then
$$
\|[b, T]_1(f_1, f_2)\|_{L^r(\R^{n+m})} \lesssim \|f_1\|_{L^p(\R^{n+m})} \|f_2\|_{L^q(\R^{n+m})}
$$
if $p,q \ne \infty$ and $r > 1$. If $T$ is free of paraproducts (so that it has a representation with shifts only),
then the same bound holds in the full range.
\end{thm}
\noindent We also obtain similar results for iterated commutators like $[b_2, [b_1, T]_1]_2$.

Regarding the extremely vast theory of commutators, we focus here only on the story of upper estimates (which are also relevant for the lower estimates)
in the multi-parameter settings.
This setting is inherently much more demanding than the one-parameter setting. For example, the lack of a satisfying theory of sparse domination, on which
many modern one-parameter proofs are based on, demands different proofs. The idea has been to rely on representation theorems such as the bi-parameter representation theorem \cite{Ma1} by one of us (or the multi-parameter generalisation of this by Y. Ou \cite{Ou}).
Ou, Petermichl and Strouse proved in \cite{OPS}  that $[b,T] \colon L^2(\R^{n+m}) \to L^2(\R^{n+m})$, when $T$ is a paraproduct
free bi-parameter singular integral. This was eventually generalised to concern all bi-parameter singular integrals satisfying $T1$ conditions
by Holmes--Petermichl--Wick \cite{HPW} -- in fact, they prove a more general Bloom type two-weight bound. For a more comprehensive account of
commutators in the multi-parameter setup see the introductions of \cite{OPS} and \cite{HPW}. 

In this paper we go after bilinear variants of these bi-parameter upper bound estimates for commutators. We point out to
the introduction of the recent paper \cite{KO} for an account of multilinear commutator estimates in the one-parameter setting.
Our proof now relies on the recent bilinear bi-parameter representation \cite{LMV}.
Compared to the linear case one of the additional difficulties lies in obtaining quasi--Banach estimates, which are in general a challenge to obtain in the bi-parameter setting even when no commutators are present:
e.g. in \cite{MPTT} -- see also \cite{MPTT2} and \cite{MS} -- the main
challange was to obtain quasi--Banach estimates for $T_m$. Moreover, bilinear model operators have more non-cancellation present,
which is a complication in the commutator setting.

The main challenge in going from \cite{OPS} to \cite{HPW} appeared to be
that estimates for $[b,S]$, where $S$ is a bi-parameter shift, were easier to obtain than for $[b, P]$, where $P$ is some other dyadic model operator (namely a full paraproduct or a partial paraproduct) appearing in the representation \cite{Ma1}. We imagine that the presence of non-cancellative Haar functions
$h_I^0$ (as opposed to cancellative Haar functions $h_I$) in the paraproducts was probably the main issue for the authors.

In the bilinear situation, however, non-cancellative Haar functions appear already in shifts. Moreover, we need an argument that can be
iterated in a reasonable way and one that can be used in restricted weak type arguments, so we needed to develop a clear general method.
Our guideline is to expand $bf$ using bi-parameter martingales in $\langle bf, h_{I} \otimes h_J\rangle$, using one-parameter martingales in
$\langle bf, h_{I}^0 \otimes h_J\rangle$ (or $\langle bf, h_I \otimes h_J^0\rangle$), and not to expand at all in $\langle bf, h_{I}^0 \otimes h_J^0\rangle$.
When working like this it appears that in the linear situation, or in the bilinear Banach range theory, there is no large difference what model operator we have, which leads to a relevant simplification. It appears to us that in \cite{HPW} everything was always reduced to a so called remainder term, which essentially entails expanding $bf$ in the bi-parameter sense
in all of the above situations. However, this remainder term has a particularly nice structure only when there are no non-cancellative Haar functions.

In the bilinear situation only when proving the Banach range boundedness are we able to obtain a unified proof that works
for all model operators. We are currently unable to produce weighted estimates for bilinear commutators, and so our quasi-Banach estimates
are now based on restricted weak type considerations. We currently only know how to do restricted weak type arguments for shifts and full paraproducts, but not for partial paraproducts. This is the case even when we are considering the operators themselves and not commutators of them. However,
in \cite{LMV} we were able to prove weighted bounds for partial paraproducts, and these can be extrapolated, so we did not require
restricted weak type arguments for partial paraproducts there. But for quasi--Banach commutator bounds we would now require them.
That is why we restrict our quasi--Banach commutator estimates to shifts, and therefore to paraproduct free singular integrals.

When running the restricted weak type argument for $[b,S]_1$, where $S$ is a bilinear bi-parameter shift,
we need to exploit the good localisation properties of bi-parameter paraproducts
as expansions of commutators essentially produce compositions of model operators and paraproducts. Furthermore, we need to be
careful so that we can move the estimates from the model operators to singular integrals, as the presence of averaging in the dyadic representation makes this a somewhat delicate business in the quasi--Banach range. Iteration requires further care to maintain the localisation properties.

We conclude by quickly giving some general references of multilinear multi-parameter analysis not connected to commutators.
For the classical linear theory of multi-parameter analysis see e.g. 
Chang and Fefferman \cite{CF1, CF2}, Fefferman \cite{Fe}, Fefferman and Stein \cite{FS}, and Journ\'e \cite{Jo1, Jo2}.
In Journ\'e \cite{Jo3} some bounds for tensor products of multilinear singular integrals are obtained.
Deep multilinear multi-parameter theory appears e.g. in the already mentioned paper Muscalu--Pipher--Tao--Thiele \cite{MPTT}, where
the quasi-Banach estimates for the multipliers $T_m$ was the main question. See also Benea--Muscalu \cite{BM1, BM2} and Di Plinio--Ou \cite{DO}.
Among many other things, these papers contain some generalisations of \cite{MPTT}, including mixed-norm type bounds. See also \cite{LMV}, where
we we proved the representation and used it to generalise many of the above results to concern completely general singular integrals. See also the book \cite{MS} by Muscalu and Schlag for a wonderful introduction to multilinear multi-parameter analysis.

\subsection*{Acknowledgements}
K. Li is supported by Juan de la Cierva - Formaci\'on 2015 FJCI-2015-24547, by the Basque Government through the BERC
2018-2021 program and by Spanish Ministry of Economy and Competitiveness
MINECO through BCAM Severo Ochoa excellence accreditation SEV-2013-0323
and through project MTM2017-82160-C2-1-P funded by (AEI/FEDER, UE) and
acronym ``HAQMEC''.

H. Martikainen is supported by the Academy of Finland through the grants 294840 and 306901, the three-year research grant 75160010 of the University of Helsinki,
and is a member of the Finnish Centre of Excellence in Analysis and Dynamics Research.

E. Vuorinen is supported by the Academy of Finland through the grant 306901 and by the Finnish Centre of Excellence in Analysis and Dynamics Research.

\section{Basic definitions}
\subsection{Vinogradov notation}
We denote $A \lesssim B$ if $A \le CB$ for some absolute constant $C$. We allow the exponent $C$ to depend on the dimension of the underlying spaces, on integration exponents, and on various other constants appearing in the assumptions. We denote $A \sim B$ if $B \lesssim A \lesssim B$.

\subsection{Dyadic notation}
If $Q$ is a cube:
\begin{itemize} 
\item $\ell(Q)$ is the side-length of $Q$;
\item $\text{ch}(Q)$ denotes the dyadic children of $Q$;
\item If $Q$ is in a dyadic grid, then $Q^{(k)}$ denotes the unique dyadic cube $S$ in the same grid so that $Q \subset S$ and $\ell(S) = 2^k\ell(Q)$;
\end{itemize}

In this paper we denote a dyadic grid in $\R^n$ by $\calD^n$ and a dyadic grid in $\R^m$ by $\calD^m$.
Using the above notation $\calD^n_i$ denotes those $I \in \calD^n$ for which $\ell(I) = 2^{-i}$. The measure of a cube $I$ is simply denoted by $|I|$ no matter in what dimension we are in.

When $I \in \calD^n$ we denote by $h_I$ a cancellative $L^2$ normalised Haar function. This means the following.
Writing $I = I_1 \times \cdots \times I_n$ we can define the Haar function $h_I^{\eta}$, $\eta = (\eta_1, \ldots, \eta_n) \in \{0,1\}^n$, by setting
\begin{displaymath}
h_I^{\eta} = h_{I_1}^{\eta_1} \otimes \cdots \otimes h_{I_n}^{\eta_n}, 
\end{displaymath}
where $h_{I_i}^0 = |I_i|^{-1/2}1_{I_i}$ and $h_{I_i}^1 = |I_i|^{-1/2}(1_{I_{i, l}} - 1_{I_{i, r}})$ for every $i = 1, \ldots, n$. Here $I_{i,l}$ and $I_{i,r}$ are the left and right
halves of the interval $I_i$ respectively. If $\eta \ne 0$ the Haar function is cancellative: $\int h_I^{\eta} = 0$. We usually suppress the presence of $\eta$
and simply write $h_I$ for some $h_I^{\eta}$, $\eta \ne 0$.

For $I \in \calD^n$ and a locally integrable function $f\colon \R^n \to \C$, we define the martingale difference
$$
\Delta_I f = \sum_{I' \in \textup{ch}(I)} \big[ \bla f \bra_{I'} -  \bla f \bra_{I} \big] 1_{I'}.
$$
Here $\bla f \bra_I = \frac{1}{|I|} \int_I f$. We also sometimes write $E_I f = \bla f \bra_I 1_I$.
Now, we have $\Delta_I f = \sum_{\eta \ne 0} \langle f, h_{I}^{\eta}\rangle h_{I}^{\eta}$, or suppressing the $\eta$ summation, $\Delta_I f = \langle f, h_I \rangle h_I$, where $\langle f, h_I \rangle = \int f h_I$.

A martingale block is defined by
$$
\Delta_{K,i} f = \mathop{\sum_{I \in \calD^n}}_{I^{(i)} = K} \Delta_I f, \qquad K \in \calD^n.
$$

\subsection{Weights}
We have some use for weighted estimates even though they are not part of the main results.
A weight $w(x_1, x_2)$ (i.e. a locally integrable a.e. positive function) belongs to $A_p(\R^n \times \R^m)$, $1 < p < \infty$, if
$$
[w]_{A_p(\R^n \times \R^m)} := \sup_{R} \bla w \bra_R \bla w' \bra_R^{p-1} < \infty,
$$
where the supremum is taken over rectangles $R \subset \R^{n+m}$ and $w' := w^{1-p'}$. We have
\begin{equation}\label{eq:prodap}
[w]_{A_p(\R^n\times \R^m)} \sim \max\big( \esssup_{x_1 \in \R^n} \,[w(x_1, \cdot)]_{A^p(\R^m)}, \esssup_{x_2 \in \R^m}\, [w(\cdot, x_2)]_{A^p(\R^n)} \big).
\end{equation}
Of course, $A_p(\R^n)$ is defined similarly as $A_p(\R^n \times \R^m)$ -- just take the supremum over cubes. For the basic theory
of bi-parameter weights consult e.g. \cite{HPW}.

\subsection{Bi-parameter notation}
We work in the bi-parameter setting in the product space $\R^{n+m}$.
In such a context $x = (x_1, x_2)$ with $x_1 \in \R^n$ and $x_2 \in \R^m$.

We often need to take integral pairings with respect to one of the two variables only.
For example, if $f \colon \R^{n+m} \to \C$, then $\langle f, h_I \rangle_1 \colon \R^{m} \to \C$ is defined by
$$
\langle f, h_I \rangle_1(x_2) = \int_{\R^n} f(y_1, x_2)h_I(y_1)\ud y_1.
$$

Next, we define bi-parameter martingale differences. Let $f \colon \R^n \times \R^m \to \C$ be locally integrable.
Let $I \in \calD^n$ and $J \in \calD^m$. We define the martingale difference
$$
\Delta_I^1 f \colon \R^{n+m} \to \C, \Delta_I^1 f(x) := \Delta_I (f(\cdot, x_2))(x_1).
$$
Define $\Delta_J^2f$ analogously, and also define $E_I^1$ and $E_J^2$ similarly.
We set
$$
\Delta_{I \times J} f \colon \R^{n+m} \to \C, \Delta_{I \times J} f(x) = \Delta_I^1(\Delta_J^2 f)(x) = \Delta_J^2 ( \Delta_I^1 f)(x).
$$
Notice that $\Delta^1_I f = h_I \otimes \langle f , h_I \rangle_1$, $\Delta^2_J f = \langle f, h_J \rangle_2 \otimes h_J$ and
$ \Delta_{I \times J} f = \langle f, h_I \otimes h_J\rangle h_I \otimes h_J$ (suppressing the finite $\eta$ summations).
We record the following standard lemma.
\begin{lem}\label{lem:mest}
For all $p \in (1,\infty)$ and $w \in A_p(\R^n \times \R^m)$ it holds that
\begin{align*}
\| f \|_{L^p(w)}
& \sim_{[w]_{A_p(\R^n \times \R^m)}} \Big\| \Big( \mathop{\sum_{I \in \calD^n}}_{J \in \calD^m} |\Delta_{I \times J} f|^2 \Big)^{1/2} \Big\|_{L^p(w)} \\
&\sim_{[w]_{A_p(\R^n \times \R^m)}}  \Big\| \Big(  \sum_{I \in \calD^n} |\Delta_I^1 f|^2 \Big)^{1/2} \Big\|_{L^p(w)} \\
&\sim_{[w]_{A_p(\R^n \times \R^m)}} \Big\| \Big(  \sum_{J \in \calD^m} |\Delta_J^2 f|^2 \Big)^{1/2} \Big\|_{L^p(w)}.
\end{align*}
\end{lem}
Martingale blocks are defined in the natural way
$$
\Delta_{K \times V}^{i, j} f  =  \sum_{I\colon I^{(i)} = K} \sum_{J\colon J^{(j)} = V} \Delta_{I \times J} f = \Delta_{K,i}^1( \Delta_{V,j}^2 f) = \Delta_{V,j}^2 ( \Delta_{K,i}^1 f).
$$

\subsection{Maximal functions}
Given dyadic grids $\calD^n$ and $\calD^m$ we denote the dyadic maximal functions
by
$$
M_{\calD^n}f(x) := \sup_{I \in \calD^n} \frac{1_I(x)}{|I|}\int_I |f(y)| \ud y
$$
and
$$
M_{\calD^n, \calD^m} f(x_1, x_2) := \sup_{R \in \calD^n \times \calD^m}  \frac{1_R(x_1, x_2)}{|R|}\iint_R |f(y_1, y_2)|\ud y_1 \ud y_2.
$$
The latter is also called the strong maximal function. The non-dyadic variants are simply denoted by $M$, as it is clear what is meant from the context.
The following definitions are in line with our usual notational conventions.
If $f \colon \R^{n+m} \to \C$ we set $M^1_{\calD^n} f(x_1, x_2)
=  M_{\calD^n}(f(\cdot, x_2))(x_1)$. The operator $M^2_{\calD^m}$ is defined similarly.
For various maximal functions $M$ we define $M_s$ by setting $M_s f = (M |f|^s)^{1/s}$.

\subsection{BMO spaces}\label{ss:bmo}

We say that $b \in L^1_{\loc}(\R^n)$ belongs to the dyadic BMO space $\BMO_{\calD^n}(\R^n) = \BMO_{\calD^n}$ if
$$
\|b\|_{\BMO_{\calD^n}} := \sup_{I \in \calD^n} \frac{1}{|I|} \int_I |b - \langle b \rangle_I| < \infty.
$$
The ordinary space $\BMO(\R^n)$ is defined by taking the supremum over all cubes.

\subsubsection*{Product BMO}
Here we define the (dyadic) bi-parameter
product BMO space $\BMO_{\textup{prod}}^{\calD^n, \calD^m}(\R^n \times \R^m) = \BMO_{\textup{prod}}^{\calD^n, \calD^m}$.
For a sequence $\lambda = (\lambda_{I,J})$ we set
\begin{equation*}
\|\lambda\|_{\BMO_{\textup{prod}}^{\calD^n, \calD^m}} := 
\sup_{\Omega} \Big( \frac{1}{|\Omega|} \mathop{\sum_{I \in \calD^n, J \in \calD^m}}_{I \times J \subset \Omega} |\lambda_{I,J}|^2 \Big)^{1/2},
\end{equation*}
where the supremum is taken over those sets $\Omega \subset \R^{n+m}$ such that $|\Omega| < \infty$ and such that for every $x \in \Omega$ there exist
$I \in \calD^n, J \in \calD^m$ so that $x \in I \times J \subset \Omega$. 

We say that $b \in L^1_{\loc}(\R^{n+m})$ belongs to the space $\BMO_{\textup{prod}}^{\calD^n, \calD^m}$ if
$$
\| b \|_{\BMO_{\textup{prod}}^{\calD^n, \calD^m}} := \| (\langle b, h_I \otimes h_J\rangle)_{I,J} \|_{\BMO_{\textup{prod}}^{\calD^n, \calD^m}} < \infty.
$$
The (non-dyadic) product BMO space $\BMO_{\textup{prod}}(\R^{n+m})$ can be defined via the norm defined by the supremum of
the above dyadic norms.

\subsubsection*{Little BMO}
We say that $b \in \bmo_{\calD^n, \calD^m}(\R^n \times \R^m) =  \bmo_{\calD^n, \calD^m}$ if
$$
\|b\|_{\bmo_{\calD^n, \calD^m}} := \sup_{\substack{I \in \calD^n \\ J \in \calD^m}} \frac{1}{|I||J|} \iint_{I \times J} |b - \langle b \rangle_{I \times J}| < \infty.
$$
The (non-dyadic) little BMO space $\bmo(\R^{n+m})$ is defined by taking the supremum over all rectangles. It is important that
$$
\|b\|_{\bmo(\R^{n+m})} \sim \max\big( \esssup_{x_1 \in \R^n} \, \|b(x_1, \cdot)\|_{\BMO(\R^m)}, \esssup_{x_2 \in \R^m}\, \|b(\cdot, x_2)\|_{\BMO(\R^n)} \big)
$$
and that we have the John--Nirenberg property
$$
\|b\|_{\bmo(\R^{n+m})} \sim \sup_{R \subset \R^{n+m}} \Big( \frac{1}{|R|} \int_{R} |b - \langle b \rangle_{R}|^p \Big)^{1/p}, \qquad 1 < p < \infty.
$$
Moreover, we need to know that $\bmo(\R^{n+m}) \subset \BMO_{\textup{prod}}(\R^{n+m})$. The reader can consult e.g. \cite{HPW, OPS}.

\subsubsection*{Adapted maximal functions}
For $b \in \BMO(\R^n)$ and $f \colon \R^n \to \C$ define
$$
M_bf = \sup_I \frac{1_I}{|I|} \int_I |b-\langle b \rangle_I| |f|.
$$
In the situation $b \in \bmo(\R^n \times \R^m)$ and $f \colon \R^{n+m} \to \C$ we similarly define
$$
M_b f = \sup_{I,J} \frac{1_{I \times J}}{|I||J|} \iint_{I \times J} |b-\langle b \rangle_{I \times J}| |f|.
$$
Here the supremums are taken over all intervals $I \subset \R^n$ and $J \subset \R^m$. The dyadic variants could also be defined, and denoted by
$M_{\calD^n, b}$ and $M_{\calD^n, \calD^m, b}$.

For a little bmo function $b \in \bmo(\R^n \times \R^m)$ define 
$$
\varphi_{\calD^m, b}(f) = \sum_{J \in \calD^m} M_{\langle b \rangle_{J,2}} \langle f, h_J \rangle_2 \otimes h_J,
$$
and similarly define $\varphi_{\calD^n, b}(f)$. For our later usage it is important to not to use the dyadic
variant $M_{\calD^n, \langle b \rangle_{J,2}}$, as it would induce an unwanted dependence on $\calD^n$ (which has relevance
in some randomisation considerations).
\begin{lem}\label{lem:bmaxbounds}
Suppose $\|b_i\|_{\BMO(\R^n)} \le 1$, $1 < u, p < \infty$ and $w \in A_p(\R^n)$. Then we have
\begin{equation}\label{eq:vMb}
\Big\| \Big( \sum_i [M_{b_i} f_i]^u \Big)^{1/u} \Big\|_{L^p(w)} \lesssim C([w]_{A_p(\R^n)}) \Big\| \Big( \sum_i |f_i|^u \Big)^{1/u} \Big\|_{L^p(w)}.
\end{equation}
The same bound holds with $\|b_i\|_{\bmo(\R^n \times \R^m)} \le 1$ and $w \in A_p(\R^n \times \R^m)$.
For $b$ with $\|b\|_{\bmo(\R^n \times \R^m)} \le 1$ we also have
$$
\|\varphi_{\calD^m, b}(f)\|_{L^p(w)} \le C([w]_{A_p(\R^n \times \R^m)})\|f\|_{L^p(w)}, \qquad 1 < p < \infty,\, w \in A_p(\R^n \times \R^m).
$$
\end{lem}
\begin{proof}
We begin by proving the bound $\|M_b f\|_{L^p(w)} \le C([w]_{A_p(\R^n \times \R^m)})\|f\|_{L^p(w)}$ -- the proof is the same in the one-parameter case.
Fix $w \in A_p(\R^n \times \R^m)$ and choose $s = s([w]_{A_p(\R^n \times \R^m)}) \in (1,p)$ so that $[w]_{A^{p/s}(\R^n \times \R^m)} \le C([w]_{A_p(\R^n \times \R^m)})$.
This can be done using the reverse H\"older inequality -- the well-known bi-parameter version is stated and proved e.g. in Proposition 2.2. of \cite{HPW}.
Using H\"older's inequality and the John--Nirenberg for little bmo we get that
$$
M_b f \le C(s) M_s f = C([w]_{A_p(\R^n \times \R^m)})M_s f.
$$
Now, using that $M \colon L^q(v) \to L^q(v)$ for all $q \in (1,\infty)$ and $v \in A_q(\R^n \times \R^m)$ we have
$$
\|M_b f\|_{L^p(w)} \le C([w]_{A_p(\R^n \times \R^m)})\| M |f|^s \|_{L^{p/s}(w)}^{1/s} \le C([w]_{A_p(\R^n \times \R^m)}) \|f\|_{L^p(w)}.
$$
The bi-parameter version of \eqref{eq:vMb} (and \eqref{eq:vMb} itself) now follow immediately by extrapolation.

Next, using Lemma \ref{lem:mest}, the estimate \eqref{eq:vMb} and the fact that $\| \langle b \rangle_{J,2} \|_{\BMO(\R^n)} \lesssim 1$
 we get
\begin{equation*}
\begin{split}
\|\varphi_{\calD^m, b}(f)\|_{L^p(w)}
&\le C([w]_{A_p(\R^n \times \R^m)}) \Big\| \Big(\sum_{J} \big(M_{\langle b \rangle_{J,2}} \langle f, h_{J}\rangle_2 \big)^2 \otimes \frac{1_{J}}{|J|} \Big)^{1/2} 
\Big\|_{L^p(w)} \\
& \le C([w]_{A_p(\R^n \times \R^m)}) \Big\| \Big(\sum_{J} \big|\langle f, h_{J}\rangle_2 \big|^2 \otimes \frac{1_{J}}{|J|} \Big)^{1/2} 
\Big\|_{L^p(w)} \\
&\le C([w]_{A_p(\R^n \times \R^m)}) \| f \|_{L^p(w)}.
\end{split}
\end{equation*}

\end{proof}

\subsection{Commutators}
We set 
$$
[b,T]_1(f_1,f_2) = bT(f_1, f_2) - T(bf_1, f_2) \, \textup{ and } \, [b,T]_2(f_1, f_2) = bT(f_1,f_2) - T(f_1, bf_2). 
$$
These are understood generally in a situation, where we e.g. already know that $T \colon L^3(\R^{n+m}) \times L^3(\R^{n+m}) \to L^{3/2}(\R^{n+m})$, and
$b$ is locally in $L^3$. Then we initially study the case that $f_1$ and $f_2$ are, say, bounded and compactly supported, so that
e.g. $bf_2 \in L^3(\R^{n+m})$ and $bT(f_1,f_2) \in L^1_{\loc}(\R^{n+m})$.

\section{Model operators: Shifts, partial paraproducts and full paraproducts}\label{sec:defbilinbiparmodel}
We define the model operators that appear
in the bilinear bi-parameter representation theorem \cite{LMV}.
In this section all the objects are defined using some fixed dyadic grids $\calD^n$ and $\calD^m$.
Let $f_1, f_2 \colon \R^{n+m} \to \C$ be two given functions. 

\subsection{Bilinear bi-parameter shifts}\label{ss:bilinbiparshiftScalar}
For triples of positive integers $k = (k_1, k_2, k_3)$, $k_1, k_2, k_3 \ge 0$, and $v = (v_1, v_2, v_3)$, $v_1, v_2, v_3 \ge 0$,
and cubes $K \in \calD^n$ and $V \in \calD^m$, define
\begin{align*}
A_{K, k}^{V, v}(f_1,f_2) = \sum_{\substack{I_1, I_2, I_3 \in \calD^n \\ I_1^{(k_1)} = I_2^{(k_2)} = I_3^{(k_3)} = K}} 
&\sum_{\substack{J_1, J_2, J_3 \in \calD^m \\ J_1^{(v_1)} = J_2^{(v_2)} = J_3^{(v_3)} = V}} a_{K, V, (I_i), (J_j)} \\
&\times \langle f_1, h_{I_1} \otimes h_{J_1}\rangle \langle f_2,  h_{I_2} \otimes h_{J_2}\rangle  h_{I_3}^0 \otimes h_{J_3}^0.
\end{align*}
We also demand that the scalars $a_{K, V, (I_i), (J_j)}$ satisfy the estimate
$$
|a_{K, V, (I_i), (J_j)}| \le \frac{|I_1|^{1/2} |I_2|^{1/2}|I_3|^{1/2}}{|K|^2} \frac{|J_1|^{1/2} |J_2|^{1/2}|J_3|^{1/2}}{|V|^2}. 
$$ 
A shift of complexity $(k,v)$ of a particular form (the non-cancellative Haar functions are in certain positions) is
$$
S_{k}^{v}(f_1,f_2) = \sum_{K \in \calD^n} \sum_{V \in \calD^m} A_{K, k}^{V, v}(f_1,f_2).
$$
An operator of the above form, but having the non-cancellative Haar functions $h_I^0$ and $h_J^0$ in some of
the other slots, is also a shift. So there are shifts of nine different types, and we could e.g. also have for all $K, V$ that
\begin{align*}
A_{K, k}^{V, v}(f_1,f_2) = \sum_{\substack{I_1, I_2, I_3 \in \calD^n \\ I_1^{(k_1)} = I_2^{(k_2)} = I_3^{(k_3)} = K}} 
&\sum_{\substack{J_1, J_2, J_3 \in \calD^m \\ J_1^{(v_1)} = J_2^{(v_2)} = J_3^{(v_3)} = V}} a_{K, V, (I_i), (J_j)} \\
&\times \langle f_1, h_{I_1}^0 \otimes h_{J_1}\rangle \langle f_2,  h_{I_2} \otimes h_{J_2}\rangle  h_{I_3} \otimes h_{J_3}^0.
\end{align*}

\subsection{Bilinear paraproducts}
Let $b \colon \R^m \to \C$ be a function and define
\begin{equation*}
A_b(g_1, g_2) := \sum_{V} \langle b, h_V\rangle \langle g_1 \rangle_V \langle g_2 \rangle_V h_V,
\end{equation*}
where $g_i \colon \R^m \to \C$. An operator $\pi_b$ is called a dyadic bilinear paraproduct in $\R^m$ if it is
of the form $A_b$, $A_b^{1*}$ or $A_b^{2*}$. We often write $\pi_{\calD^m ,b}$ to emphasise the dyadic grid using which
it is defined.

\subsection{Bilinear bi-parameter partial paraproducts}
Let $k = (k_1, k_2, k_3)$, $k_1, k_2, k_3 \ge 0$. For each $K, I_1, I_2, I_3 \in \calD^n$ we are given a function $b_{K, I_1, I_2, I_3} \colon \R^m \to \C$ such that
$$
\| b_{K, I_1, I_2, I_3} \|_{\BMO(\R^m)} \le \frac{|I_1|^{1/2} |I_2|^{1/2}|I_3|^{1/2}}{|K|^2}.
$$
A partial paraproduct of complexity $k$ of a particular form is
$$
P_k(f_1, f_2) = \sum_{K \in \calD^n} \sum_{\substack{I_1, I_2, I_3 \in \calD^n \\ I_1^{(k_1)} = I_2^{(k_2)} = I_3^{(k_3)} = K}} h_{I_3}^0 \otimes
\pi_{b_{K, I_1, I_2, I_3}}(\langle f_1, h_{I_1} \rangle_1, \langle f_2, h_{I_2} \rangle_1),
$$ 
where $\pi_{b_{K, I_1, I_2, I_3}} $denotes a bilinear paraproduct in $\R^m$, and is of the same form for all $K,I_1, I_2, I_3$.
Again, an operator of the above form, but having the non-cancellative Haar function $h_I^0$ in some other slot, is also a partial paraproduct. 
Therefore, we have nine different possibilities again (the bilinear paraproducts can be of one of the three different types, and the non-cancellative Haar function in $\R^n$ can appear in one of the three slots).

Of course, we also have partial paraproducts with shift structure in $\R^m$ and paraproducts in $\R^n$.

\subsection{Bilinear bi-parameter full paraproducts}
Given a function $b \colon \R^{n+m} \to \C$ with $\|b\|_{\BMO_{\textup{prod}}(\R^{n+m})} = 1$ a full paraproduct $\Pi_b$ of a particular form is
$$
\Pi_b(f_1, f_2) = \sum_{\substack{K \in \calD^n \\ V \in \calD^m}} \lambda_{K,V}^b \langle f_1 \rangle_{K \times V} \langle f_2 \rangle_{K \times V} h_K \otimes h_V,
$$
where the function $b$ determines the coefficients $\lambda_{K,V}^b$ via the formula $\lambda_{K,V}^b := \langle b, h_K \times h_V \rangle$.
Again, an operator of the above form, but having the cancellative Haar functions $h_K$ or $h_V$ in some other slots, is also a full paraproduct. There are nine
different cases as the Haar functions present in the coefficients $\lambda_{K,V}^b$ are not allowed to move, i.e. we always have $\lambda_{K,V}^b := \langle b, h_K \times h_V \rangle$. For example, $\Pi_b$ could also be of the form
$$
\Pi_b(f_1, f_2) = \sum_{\substack{K \in \calD^n \\ V \in \calD^m}} \lambda_{K,V}^b \langle f_1 \rangle_{K \times V} \Big\langle f_2, \frac{1_K}{|K|} \otimes h_V \Big\rangle h_K \otimes \frac{1_V}{|V|}.
$$
\begin{rem}
We warn the reader that later we will have \emph{linear} bi-parameter paraproducts (the operators $A_i(b, \cdot)$, $i = 5,6,7,8$, in Section \ref{sec:marprod})
so that even the coefficients $\lambda_{K,V}^b$ can have $\frac{1_K}{|K|} \otimes h_V$ or $h_K \otimes \frac{1_V}{|V|}$. The role of such paraproducts is the following: they
appear in some decompositions of $bf$ related to commutators, but they \emph{do not} appear in the linear bi-parameter representation theorem \cite{Ma1}.
In fact, their boundedness also requires more: $b$ has to be in $\bmo(\R^n \times \R^m)$.
In this section we are only introducing operators that appear in the bilinear bi-parameter representation theorem \cite{LMV}, so philosophies
of such nature do not concern us here. 
\end{rem}
\subsection{Boundedness properties of the model operators}
In \cite{LMV} we showed that all the model operators are bounded in the full range $L^p(\R^{n+m}) \times L^q(\R^{n+m}) \to L^r(\R^{n+m})$,
$p,q \in (1,\infty]$, $r \in (1/2, \infty)$ and $1/p + 1/q = 1/r$. In fact, we even showed various weighted estimates and mixed-norm estimates.

\section{Bilinear bi-parameter singular integrals and commutators}
A bilinear bi-parameter singular integral $T$ has a relatively long definition. A model of a bilinear bi-parameter CZO in $\R^n \times \R^m$ is
$$
(T_1 \otimes T_2)(f_1 \otimes f_2, g_1 \otimes g_2)(x) := T_1(f_1, g_1)(x_1)T_2(f_2, g_2)(x_2),
$$
where $f_1, g_1 \colon \R^n \to \C$, $f_2, g_2 \colon \R^m \to \C$,
$x = (x_1, x_2) \in \R^{n+m}$, $T_1$ is a bilinear CZO in $\R^n$ and $T_2$ is a bilinear CZO in $\R^m$. 
A model of a bilinear CZO $T$ in $\R^n$ is
$$
T(f_1, f_2)(x) := \tilde T(f_1 \otimes f_2)(x,x), \qquad x \in \R^n,
$$
where $\tilde T$ is a usual linear CZO in $\R^{2n}$. For the general definition of a bilinear singular integral see e.g. \cite{GT}.
For the general definition of bilinear bi-parameter singular integrals
we refer to Section 3 of \cite{LMV}. In \cite{LMV} we proved that under certain natural T1 type conditions we can represent
$\langle T(f_1, f_2), f_3 \rangle$ using the model operators from Section \ref{sec:defbilinbiparmodel} (shifts, partial paraproducts and full paraproducts). For the definition of the
$T1$ type conditions (their exact nature is not needed in this paper) we again refer to Section 3 of \cite{LMV}.

We now state the bilinear bi-parameter representation theorem from Section 5 of \cite{LMV}. For this we need the following notation regarding random dyadic grids.
Let $\mathcal{D}_0^n$ and $\mathcal{D}_0^m$ denote the standard dyadic grids on $\R^n$ and $\R^m$ respectively.
For $\omega = (\omega_i) \in (\{0,1\}^n)^{\Z}$, $\omega' = ( \omega'_i) \in(\{0,1\}^m)^{\Z}$, $I \in \calD^n_0$ and $J \in \calD^m_0$ denote
$$
I + \omega := I + \sum_{i:\, 2^{-i} < \ell(I)} 2^{-i}\omega_i \qquad \textup{and} \qquad J + \omega' := J + \sum_{i:\, 2^{-i} < \ell(J)} 2^{-i}\omega'_i.
$$
Then we define the random lattices
$$\calD^n_{\omega} = \{I + \omega\colon I \in \calD^n_0\} \qquad \textup{and} \qquad
\calD^m_{\omega'} = \{J + \omega'\colon J \in \calD^m_0\}.
$$
In what follows always $\omega \in (\{0,1\}^n)^{\Z}$ and  $\omega' \in(\{0,1\}^m)^{\Z}$.
There is a natural probability product measure $\mathbb{P}_{\omega}$ in $(\{0,1\}^n)^{\Z}$ and $\mathbb{P}_{\omega'}$ in $(\{0,1\}^m)^{\Z}$.
We denote the expectation over these probability spaces by $\E_{\omega, \omega'} = \E_{\omega} \E_{\omega'} = \iint \ud \mathbb{P}_{\omega} \ud \mathbb{P}_{\omega'}$.

We sometimes can also write $\calD_0 = \calD^n_0 \times \calD^m_0$ and $\calD_{\omega, \omega'} = \calD^n_{\omega} \times \calD^m_{\omega'}$.
\begin{thm}\label{thm:rep}
Suppose $T$ is a bilinear bi-parameter Calder\'on--Zygmund operator satisfying all the structural assumptions and all the
boundedness and cancellation assumptions as formulated in Section 3 of \cite{LMV}. Then
$$
\langle T(f_1,f_2), f_3\rangle = C_T \mathbb{E}_{\omega, \omega'}\mathop{\sum_{k = (k_1, k_2, k_3) \in \Z_+^3}}_{v = (v_1, v_2, v_3) \in \Z_+^3} \alpha_{k, v} 
\sum_{u}
\bla U^{v}_{k, u, \mathcal{D}^n_{\omega},\mathcal{D}^m_{\omega'}}(f_1, f_2), f_3 \bra,
$$
where $C_T \lesssim 1$, $\alpha_{k, v} = 2^{- \alpha \max k_i/2} 2^{- \alpha \max v_j/2}$, the summation over $u$ is finite, and
$U^{v}_{k, u, \mathcal{D}^n_{\omega},\mathcal{D}^m_{\omega'}}$ is always either a shift of complexity $(k,v)$,
a partial paraproduct of complexity $k$ or $v$ (this requires $k= 0$ or $v=0$) or a full paraproduct (this requires $k=v=0$).
We can e.g. understand that here $f_1, f_2, f_3 \in L^3(\R^{n+m})$.
\end{thm}
We recall that in \cite{LMV} we in particular showed that every bilinear bi-parameter singular integral $T$
satisfying the assumptions of the above representation theorem maps in the full range $L^p(\R^{n+m}) \times L^q(\R^{n+m}) \to L^r(\R^{n+m})$,
$p,q \in (1,\infty]$, $r \in (1/2, \infty)$ and $1/p + 1/q = 1/r$. In fact, we showed much more general bounds -- see \cite{LMV}.

We can now formulate our theorem about the Banach range boundedness of $[b,T]_1$, where $T$ is a bilinear bi-parameter singular integral
satisfying the assumptions of the above representation theorem and $\|b\|_{\bmo(\R^{n+m})}  = 1$.
\begin{thm}\label{thm:main1}
Suppose $T$ is a bilinear bi-parameter singular integral satisfying the assumptions of Theorem \ref{thm:rep} and $\|b\|_{\bmo(\R^{n+m})}  = 1$.
Let $p,q,r \in (1,\infty)$ with $1/p + 1/q = 1/r$. Then we have
$$
\|[b, T]_1(f_1, f_2)\|_{L^r(\R^{n+m})} \lesssim \|f_1\|_{L^p(\R^{n+m})} \|f_2\|_{L^q(\R^{n+m})}.
$$
\end{thm}
\begin{proof}
The claim follows from Theorem \ref{thm:rep} and from the Banach range commutator bounds of the model operators, Theorem \ref{thm:com1ofmodelBanach}.
\end{proof}

A bilinear bi-parameter singular integral $T$ is called free of paraproducts if
for all suitable functions $f_i \colon \R^n \to \C$, $g_i \colon \R^m \to \C$, $i=1,2$,
all 
$$
S \in \{T, T^{*1}, T^{*2}, T^{1*}_1, T^{2*}_1, T^{1*}_2, T^{2*}_2, T^{1*, 2*}_{1,2}, T^{1*, 2*}_{2,1}\}
$$
and all cubes $I \subset \R^n$, $J \subset \R^m$ there holds
$$
\langle S(1 \otimes g_1, 1 \otimes g_2), h_I \otimes h_J \rangle 
= \langle S(f_1 \otimes 1, f_2 \otimes 1), h_I \otimes h_J \rangle =0.
$$
For the definition of all the nine adjoints and partial adjoints of T see Section 2.8 of \cite{LMV}.
This definition guarantees that $T$ has a representation with shifts only.
In Section 8 of \cite{LMV} we proved that the bi-parameter multipliers $T_m$ of \cite{MPTT}
are bilinear bi-parameter singular integrals satisfying the assumptions of Theorem \ref{thm:rep} and
that they are free of paraproducts.
Our second main theorem involving quasi--Banach estimates for commutators of paraproduct free singular integrals is:
\begin{thm}\label{thm:main2}
Suppose $T$ is a bilinear bi-parameter singular integral satisfying the assumptions of Theorem \ref{thm:rep} and
that $T$ is free of paraproducts. Let $\|b\|_{\bmo(\R^{n+m})}  = 1$, and
let $1 < p, q \le \infty$ and $1/2 < r < \infty$ satisfy $1/p+1/q = 1/r$.
Then we have
$$
\|[b, T]_1(f_1, f_2)\|_{L^r(\R^{n+m})} \lesssim \|f_1\|_{L^p(\R^{n+m})} \|f_2\|_{L^q(\R^{n+m})}.
$$
\end{thm}
\begin{proof}
Using Theorem \ref{thm:rep} write the pointwise identity
$$
[b,T]_1(f_1, f_2) = C_T \mathop{\sum_{k = (k_1, k_2, k_3) \in \Z_+^3}}_{v = (v_1, v_2, v_3) \in \Z_+^3} \alpha_{k, v} 
\sum_{u} \mathbb{E}_{\omega, \omega'} [b,S^{v}_{k, u, \mathcal{D}^n_{\omega},\mathcal{D}^m_{\omega'}}]_1(f_1, f_2),
$$
where $S^{v}_{k, u, \mathcal{D}^n_{\omega},\mathcal{D}^m_{\omega'}}$ are bilinear bi-parameter shifts
of complexity $(k,v)$ defined using the dyadic grids $\mathcal{D}^n_{\omega}$ and $\mathcal{D}^m_{\omega'}$.
If $r \in (1/2,1]$ we get using $\|\sum_i g_i\|_{L^r(\R^{n+m})}^r \le \sum_i \| g_i  \|_{L^r(\R^{n+m})}^r$ that we have
$$
\| [b,T]_1(f_1, f_2) \|_{L^r(\R^{n+m})}^r \lesssim \mathop{\sum_{k = (k_1, k_2, k_3) \in \Z_+^3}}_{v = (v_1, v_2, v_3) \in \Z_+^3} \alpha_{k, v}^r
\sum_{u}  \|\mathbb{E}_{\omega, \omega'}[b,S^{v}_{k, u, \mathcal{D}^n_{\omega},\mathcal{D}^m_{\omega'}}]_1(f_1, f_2)\|_{L^r(\R^{n+m})}^r.
$$
If $r > 1$ simply use $\|\sum_i g_i\|_{L^r(\R^{n+m})} \le \sum_i \| g_i  \|_{L^r(\R^{n+m})}$ instead.
The claim then follows using Theorem \ref{thm:com1ofmodelQuasiBanach}, which says that \emph{averages} of commutators of shifts
map in the full range with a bound polynomial in complexity.
\end{proof}
The corresponding results for iterated commutators are recorded and proved in Section \ref{sec:iterated}.

\section{Martingale difference expansions of products}\label{sec:marprod}
The idea is that a product $bf$ paired with Haar functions is expanded in the bi-parameter fashion only if both of the Haar functions are cancellative. In a mixed
situation we expand only in $\R^n$ or $\R^m$, and in the remaining fully non-cancellative situation we do not expand at all -- and this protocol is key for us.
Also, our protocol entails the following: when pairing with a non-cancellative Haar function we add and subtract a suitable average of $b$.

Let $\calD^n$ and $\calD^m$ be some fixed dyadic grids in $\R^n$ and $\R^m$, respectively, and write $\calD= \calD^n \times \calD^m$.
In what follows we sum over $I \in \calD^n$ and $J \in \calD^m$.
\subsection{Paraproduct operators}
Let us first define certain paraproduct operators:
\begin{align*}
A_1(b,f) &= \sum_{I, J} \Delta_{I \times J} b \Delta_{I \times J} f, \\
A_2(b,f) &= \sum_{I, J} \Delta_{I \times J} b E_I^1\Delta_J^2 f, \\
A_3(b,f) &= \sum_{I, J} \Delta_{I \times J} b \Delta_I^1 E_J^2  f,\\
A_4(b,f) &= \sum_{I, J} \Delta_{I \times J} b \bla f \bra_{I \times J},
\end{align*}
and
\begin{align*}
A_5(b,f) &= \sum_{I, J} E_I^1 \Delta_J^2 b \Delta_{I \times J} f, \\
A_6(b,f) &= \sum_{I, J}  E_I^1 \Delta_J^2 b  \Delta_I^1 E_J^2  f, \\
A_7(b,f) &= \sum_{I, J} \Delta_I^1 E_J^2  b \Delta_{I \times J} f, \\
A_8(b,f) &= \sum_{I, J}  \Delta_I^1 E_J^2 b E_I^1 \Delta_J^2 f.
\end{align*}
We grouped these into two collections, because these are handled differently.

When desired, these operators can be written with Haar functions using 
$$\Delta_{I \times J} g = \sum_{I, J} \langle g, h_I \otimes h_J\rangle h_I \otimes h_J,  \,\,
\Delta_I^1 g = h_I \otimes \langle g, h_I \rangle_1 \textup{ and } \Delta_J^2 g = \langle g, h_J \rangle_2 \otimes h_J,
$$
where we have suppressed the signatures $h_I = h_I^{\epsilon}$ and $h_J = h_J^{\delta}$, $\epsilon \in \{0,1\}^n \setminus \{0\}$, $\delta \in \{0,1\}^m \setminus \{0\}$, of the Haar functions.
This means that the finite summations over the signatures are implicitly understood. To understand things correctly, one has to be slightly careful when a term like $h_I h_I$ or $h_J h_J$ appears (as they do e.g. when expanding $A_1$ using Haar functions).
This really can be of the form $h_I^{\epsilon_1} h_I^{\epsilon_2}$ for possibly different $\epsilon_1, \epsilon_2$. However, the only property we will
use is that $|h_I h_I| = 1_I / |I|$, i.e. we always treat such products as non-cancellative objects (the available cancellation when $\epsilon_1 \ne \epsilon_2$
is simply never needed or used).

Suppose that $\|b\|_{\BMO_{\textup{prod}}(\R^{n+m})} = 1$. Then standard theory tells us that for $i = 1, \ldots, 4$ we have
\begin{equation}\label{eq:wforlargeA1}
\|A_i(b, f)\|_{L^p(w)} \lesssim C([w]_{A_p(\R^n \times \R^m)}) \|f\|_{L^p(w)}, \, p \in (1,\infty), \, w \in A_p(\R^n \times \R^m).
\end{equation}
This is because these operators are bi-parameter paraproducts, and their boundedness follows easily by using
$$
\sum_{I, J} |\langle b, h_I \otimes h_J \rangle| |A_{IJ}| \lesssim \Big\|\Big( \sum_{I,J} |A_{IJ}|^2 \frac{1_{I \times J}}{|I \times J|} \Big)^{1/2} \Big\|_{L^1(\R^{n+m})}.
$$
For a simple proof of this inequality see e.g. Proposition 4.1 of \cite{MO}.

If we assume more in that $\|b\|_{\bmo(\R^{n+m})} = 1$, then also for $i = 5, \ldots, 8$ we have
\begin{equation}\label{eq:wforlargeA2}
\|A_i(b, f)\|_{L^p(w)} \lesssim C([w]_{A_p(\R^n \times \R^m)}) \|f\|_{L^p(w)}, \, p \in (1,\infty), \, w \in A_p(\R^n \times \R^m).
\end{equation}
The proofs of these bounds are similar to the above ones, and are proved using that uniformly on $I$ we have
$$
\sum_J \Big| \Big\langle b, \frac{1_I}{|I|} \otimes h_J\Big\rangle\Big| |A_{IJ}| \lesssim \Big\|\Big( \sum_{J} |A_{IJ}|^2 \frac{1_J}{|J|} \Big)^{1/2} \Big\|_{L^1(\R^m)}.
$$

We also define
$$
a^1_1(b,f) = \sum_I \Delta_I^1 b \Delta_I^1 f
$$
and
$$
a^1_2(b,f) = \sum_I \Delta_I^1 b E_I^1 f.
$$
Again, we have that if $\|b\|_{\bmo(\R^{n+m})} = 1$ then for $i=1,2$ we have
\begin{equation}\label{eq:wforsmallA}
\|a_i^1(b, f)\|_{L^p(w)} \lesssim C([w]_{A_p(\R^n \times \R^m)}) \|f\|_{L^p(w)}, \,\, p \in (1,\infty), \, w \in A_p(\R^n \times \R^m).
\end{equation}
The operators $a^2_1(b,f)$ and $a^2_2(b,f)$ are defined analogously.

Let now $I_0 \in \calD^n$ and $ J_0 \in \calD^m$, and suppose $b \in \bmo(\R^{n+m})$ and  $f\in L^{p_0}(\R^{n+m})$ for some $p_0 \in (1, \infty)$.
We introduce our basic expansions of $\langle bf, h_{I_0} \otimes h_{J_0}\rangle$ and $\langle bf, h_{I_0} \otimes h_{J_0}^0\rangle$ (the expansion
in the case $\langle bf, h_{I_0}^0 \otimes h_{J_0}\rangle$ being symmetric, of course).
\subsection{Expansion of $\langle bf, h_{I_0} \times h_{J_0} \rangle$}
We know that $b \in L^{p_0'}_{\loc}(\R^{n+m})$. Therefore,
there holds
$$
1_{I_0 \times J_0} b 
= \sum_{\substack{I_1\times J_1 \in \calD \\ I_1 \times J_1 \subset I_0 \times J_0}}\Delta_{I_1 \times J_1} b
+\sum_{\substack{J_1 \in \calD^m \\ J_1 \subset J_0}} E^1_{I_0} \Delta^2_{J_1} b
+ \sum_{\substack{I_1 \in \calD^n \\ I_1 \subset I_0}} \Delta^1_{I_1} E^2_{J_0} b
+ E_{I_0 \times J_0} b.
$$
Let us denote these terms by $I_i$, $i=1,2,3,4$, in the respective order.
We have the corresponding decomposition of $f$, whose terms we denote by $II_i$, $i=1,2,3,4$. Notice that
$$
\sum_{i=1}^4 \langle I_1 II_i, h_{I_0} \otimes h_{J_0} \rangle = \sum_{i=1}^4 \langle A_i(b, f), h_{I_0} \otimes h_{J_0} \rangle,
$$
$$
\sum_{i=1}^4 \langle I_2 II_i, h_{I_0} \otimes h_{J_0} \rangle = \sum_{i=5}^6 \langle A_i(b, f), h_{I_0} \otimes h_{J_0} \rangle,
$$
$$
\sum_{i=1}^4 \langle I_3 II_i, h_{I_0} \otimes h_{J_0} \rangle = \sum_{i=7}^8 \langle A_i(b, f), h_{I_0} \otimes h_{J_0} \rangle
$$
and
$$
\sum_{i=1}^4 \langle I_4 II_i, h_{I_0} \otimes h_{J_0} \rangle = \langle b \rangle_{I_0 \times J_0} \langle f, h_{I_0} \otimes h_{J_0} \rangle.
$$
Therefore, we have
\begin{equation}\label{eq:biparEX}
\langle bf, h_{I_0} \otimes h_{J_0} \rangle = \sum_{i=1}^8 \langle A_i(b, f), h_{I_0} \otimes h_{J_0} \rangle + \langle b \rangle_{I_0 \times J_0} \langle f, h_{I_0} \otimes h_{J_0} \rangle.
\end{equation}

\subsection{Expansion of $\langle bf, h_{I_0} \times h_{J_0}^0 \rangle$}
This time we write
$$
1_{I_0} b  = \sum_{\substack{I_1 \in \calD^n \\ I_1 \subset I_0}}\Delta_{I_1}^1 b + E_{I_0}^1 b,
$$
and similarly for $f$, and notice that
$$
\langle bf, h_{I_0} \rangle_1 = \sum_{i=1}^2 \langle a_i^1(b,f), h_{I_0}\rangle_1 + \langle b \rangle_{I_0,1} \langle f, h_{I_0}\rangle_1.
$$
Therefore, we have
\begin{equation}\label{eq:1EX}
\begin{split}
\langle bf, h_{I_0} \otimes h_{J_0}^0 \rangle &= \sum_{i=1}^2 \langle a_i^1(b,f), h_{I_0} \otimes h_{J_0}^0 \rangle \\
&+ \langle (\langle b \rangle_{I_0,1} - \langle b \rangle_{I_0 \times J_0}) \langle f, h_{I_0}\rangle_1, h_{J_0}^0\rangle
+ \langle b \rangle_{I_0 \times J_0} \langle f, h_{I_0} \otimes h_{J_0}^0 \rangle.
\end{split}
\end{equation}

When we have $\langle bf, h_{I_0}^0 \otimes h_{J_0}^0 \rangle$ we do not expand at all, we simply add and subtract an average:
\begin{equation}\label{eq:noEX}
\langle bf, h_{I_0}^0 \otimes h_{J_0}^0 \rangle = \langle (b-\langle b \rangle_{I_0 \times J_0})f, h_{I_0}^0 \otimes h_{J_0}^0 \rangle
+ \langle b \rangle_{I_0 \times J_0} \langle f, h_{I_0}^0 \otimes h_{J_0}^0 \rangle.
\end{equation}

\section{Key identities related to commutators}\label{sec:KeyIdentities}
We state some lemmas related to identities that appear when we expand using
\eqref{eq:biparEX}, \eqref{eq:1EX}, the symmetric form of \eqref{eq:1EX} or \eqref{eq:noEX} in commutators of model operators.
The proofs of these lemmas are trivial applications of these identities, and the cancellation present in the commutators
is simply exploited by grouping the terms involving free averages of $b$ together (the last term appearing in these identities).

In the first order commutators of model operators there are essentially seven different symmetries
depending on how many non-cancellative Haar functions we have in the model operator in question, and how they are situated -- for these symmetries see the proof of Theorem \ref{thm:com1ofmodelBanach}. We only explicitly state lemmas relevant for three of these symmetries, but the remaining
identities are completely analogous and obtained by expanding using the described protocol.

Below we have $I,Q \in \calD^n$ and $J, R \in \calD^m$ with some fixed dyadic grids $\calD^n$ and $\calD^m$.
\begin{lem}\label{lem:case1}
We have
\begin{align*}
\langle f, h_I \otimes h_J& \rangle \langle bg, h_Q \otimes h_R \rangle - \langle bf, h_I \otimes h_J \rangle \langle g, h_Q \otimes h_R \rangle \\
&= \sum_{i=1}^8 \langle f, h_I \otimes h_J \rangle \langle A_i(b,g), h_Q \otimes h_R \rangle \\
&- \sum_{i=1}^8 \langle A_i(b,f), h_I \otimes h_J \rangle \langle g, h_Q \otimes h_R \rangle \\
&+ [\langle b \rangle_{Q \times R} - \langle b \rangle_{I \times J}] \langle f, h_I \otimes h_J \rangle \langle g, h_Q \otimes h_R \rangle.
\end{align*}
\end{lem}

\begin{lem}\label{lem:case2}
We have
\begin{align*}
\langle f, h_I^0 \otimes h_J& \rangle \langle bg, h_Q \otimes h_R \rangle - \langle bf, h_I^0 \otimes h_J \rangle \langle g, h_Q \otimes h_R \rangle \\
&= \sum_{i=1}^8 \langle f, h_I^0 \otimes h_J \rangle \langle A_i(b,g), h_Q \otimes h_R \rangle \\
&- \sum_{i=1}^2 \langle a_i^2(b,f), h_I^0 \otimes h_J \rangle \langle g, h_Q \otimes h_R \rangle \\
&+ \bla (\langle b \rangle_{I \times J} - \langle b \rangle_{J, 2})\langle f, h_J\rangle_2, h_{I}^0\bra \langle g, h_Q \otimes h_R \rangle \\
&+  [\langle b \rangle_{Q \times R} - \langle b \rangle_{I \times J}] \langle f, h_I^0 \otimes h_J \rangle \langle g, h_Q \otimes h_R \rangle
\end{align*}
\end{lem}

\begin{lem}\label{lem:case3}
We have
\begin{align*}
\langle f, h_I^0 \otimes h_J& \rangle \langle bg, h_Q \otimes h_R^0 \rangle - \langle bf, h_I^0 \otimes h_J \rangle \langle g, h_Q \otimes h_R^0 \rangle \\
&=\sum_{i=1}^2  \langle f, h_I^0 \otimes h_J \rangle \langle a^1_i (b,g), h_Q \otimes h_R^0 \rangle \\
&-\sum_{i=1}^2 \langle a^2_i(b,f), h_I^0 \otimes h_J \rangle \langle g, h_Q \otimes h_R^0 \rangle \\
&+\langle f, h_I^0 \otimes h_J \rangle \bla (\langle b \rangle_{Q,1}-\langle b \rangle_{Q \times R}) \langle g ,h_Q \rangle_1, h^0_R \bra \\
&-\bla (\langle b \rangle_{J,2}-\langle b \rangle_{I \times J}) \langle f,h_J \rangle_2,h^0_I \bra \langle g, h_Q \otimes h_R^0 \rangle \\
&+[\langle b \rangle_{Q \times R} - \langle b \rangle_{I \times J}] \langle f, h_I^0 \otimes h_J \rangle \langle g, h_Q \otimes h_R^0 \rangle.
\end{align*}
\end{lem}

We also record two additional lemmas, which are used in conjunction with such identities.
\begin{lem}\label{lem:maximalbound}
For $I \in \calD^n$ and $J \in \calD^m$ we have
$$
\big|\bla (\langle b \rangle_{J, 2} -\langle b \rangle_{I \times J} )\langle f, h_J\rangle_2 \bra_I\big| 
\le \Big\langle \varphi_{\calD^m, b}(f), \frac{1_I}{|I|} \otimes h_J\Big\rangle
$$
and
$$
\big|\bla (b-\langle b \rangle_{I \times J})f \bra_{I \times J}\big| \lesssim \langle  M_b f \rangle_{I \times J}.
$$
\end{lem}
\begin{proof}
There holds
\begin{equation*}
\big|\bla (\langle b \rangle_{J, 2} -\langle b \rangle_{I \times J} )\langle f, h_J\rangle_2 \bra_I\big| \le \langle  M_{\langle b \rangle_{J,2}} \langle f, h_J\rangle_2 \rangle_{I}
= \Big\langle \varphi_{\calD^m, b}(f), \frac{1_I}{|I|} \otimes h_J\Big\rangle,
\end{equation*}
where the last inequality follows from orthogonality. The second claimed inequality is even more immediate.
\end{proof}

\begin{lem}\label{lem:bmobound}
Suppose $I^{(i)} = Q^{(q)} = K$ and $J^{(j)} = R^{(r)} = V$. If $\|b\|_{\bmo(\R^{n+m})} = 1$ then we have
$$
|\langle b \rangle_{Q \times R} - \langle b \rangle_{I \times J}| \lesssim \max(i, j, q, r).
$$
\end{lem}
\begin{proof}
Estimate
$$
|\langle b \rangle_{Q \times R} - \langle b \rangle_{I \times J}| \le |\langle b \rangle_{Q \times R} - \langle b \rangle_{K \times V}|
+ |\langle b \rangle_{I \times J}- \langle b \rangle_{K \times V}|,
$$
and use repeatedly that
$$
|\langle b \rangle_{Q \times R} - \langle b \rangle_{Q^{(1)} \times R}| \le \langle |b- \langle b \rangle_{Q^{(1)} \times R}| \rangle_{Q \times R}
\lesssim \langle |b- \langle b \rangle_{Q^{(1)} \times R}| \rangle_{Q^{(1)} \times R} \le 1.
$$
\end{proof}

\section{Banach range boundedness of commutators of model operators}\label{sec:BanachforModels}
For the Banach range theory of commutators we only need the rather easy fact that all the model operators from Section \ref{sec:defbilinbiparmodel}
are of the following general type. Fix dyadic grids $\calD^n$ and $\calD^m$.
Let $U = U^v_k$, $0 \le k_i \in \Z$ and $0 \le v_i \in \Z$, $i=1,2,3$, be a bilinear bi-parameter operator such that 
\begin{equation*}
\begin{split}
\langle U(f_1,f_2),f_3 \rangle
= \sum_{\substack{K \in \calD^n \\ V \in \calD^m}} 
\sum_{\substack{I_1, I_2, I_3 \in \calD^n \\ I_1^{(k_1)} = I_2^{(k_2)} = I_3^{(k_3)} = K}} 
&\sum_{\substack{J_1, J_2, J_3 \in \calD^m \\ J_1^{(v_1)} = J_2^{(v_2)} = J_3^{(v_3)} = V}} a_{K, V, (I_i), (J_j)} \\
&\times \langle f_1, \wt h_{I_1} \otimes  \wt h_{J_1}\rangle \langle f_2, \wt h_{I_2} \otimes \wt h_{J_2}\rangle  
\langle f_3, \wt  h_{I_3} \otimes \wt h_{J_3} \rangle,
\end{split}
\end{equation*}
where $a_{K, V, (I_i), (J_j)}$ are constants and for all $i=1,2,3$ we have $ \wt h_{I_i}= h_{I_i}$ for all $I_i \in \calD^n$ or $\wt h_{I_i}= h_{I_i}^0$ for all $I_i\in \calD^n$,
and similarly with the functions $\wt h_{J_j}$. We assume that for all $p,q,r \in (1,\infty)$ with $1/p + 1/q = 1/r$ we have
\begin{equation}\label{eq:bb}
\begin{split}
\sum_{\substack{K \in \calD^n \\ V \in \calD^m}} 
\sum_{\substack{I_1, I_2, I_3 \in \calD^n \\ I_1^{(k_1)} = I_2^{(k_2)} = I_3^{(k_3)} = K}} 
&\sum_{\substack{J_1, J_2, J_3 \in \calD^m \\ J_1^{(v_1)} = J_2^{(v_2)} = J_3^{(v_3)} = V}} \big| a_{K, V, (I_i), (J_j)} \\
&\times \langle f_1, \wt h_{I_1} \otimes  \wt h_{J_1}\rangle \langle f_2, \wt h_{I_2} \otimes \wt h_{J_2}\rangle  
\langle f_3, \wt  h_{I_3} \otimes \wt h_{J_3} \rangle \big | \\
& \lesssim \| f_1 \|_{L^p} \| f_2 \|_{L^q}\| f_3 \|_{L^{r'}}. 
\end{split}
\end{equation}
We do not assume anything else about the constants $a_{K, V, (I_i), (J_j)}$.
In particular, $U$ can be a bilinear bi-parameter shift, a partial paraproduct or a full paraproduct.

\begin{thm}\label{thm:com1ofmodelBanach}
Let $p,q,r \in (1,\infty)$, $1/p + 1/q = 1/r$, $0 \le k_i \in \Z$ and $0 \le v_i \in \Z$, $i=1,2,3$. 
Let $U = U^v_k$ be a general bilinear bi-parameter model operator satisfying
\eqref{eq:bb}. In particular, $U$ can be a bilinear bi-parameter shift, a partial paraproduct or a full paraproduct.
Then for $b$ such that $\|b\|_{\bmo(\R^{n+m})} = 1$ we have
$$
\|[b, U]_1(f_1, f_2)\|_{L^r(\R^{n+m})} \lesssim (1+\max(k_i, v_i)) \|f_1\|_{L^p(\R^{n+m})} \|f_2\|_{L^q(\R^{n+m})}.
$$

\end{thm}

\begin{proof}
 We separately treat the different possible combinations of cancellative and non-cancellative Haar functions.
 The proof depends only on what Haar functions  we have paired with $f_1$ and $f_3$ in $\langle U(f_1, f_2), f_3\rangle$.
 All the model operators fall into one of the following cases:

\begin{enumerate}
\item  We have $\langle f_1,  h_{I_1} \otimes  h_{J_1}\rangle
\langle f_3,  h_{I_3} \otimes  h_{J_3}\rangle$.

\item   We have $\langle f_1, h_{I_1}^0 \otimes h_{J_1}\rangle \langle f_3,  h_{I_3} \otimes  h_{J_3}\rangle$, or one of the three other symmetric
cases. 

\item  We have $\langle f_1,  h_{I_1}^0 \otimes  h_{J_1}\rangle
\langle f_3,  h_{I_3} \otimes  h_{J_3}^0\rangle$, or the symmetric case.

\item We have $\langle f_1,  h_{I_1}^0 \otimes  h_{J_1}^0\rangle
\langle f_3,  h_{I_3} \otimes  h_{J_3}\rangle$, or the symmetric case.

\item We have $\langle f_1,  h_{I_1}^0 \otimes  h_{J_1}\rangle
\langle f_3,  h_{I_3}^0 \otimes  h_{J_3}\rangle$, or the symmetric case.

\item We have  $\langle f_1,  h_{I_1}^0 \otimes  h_{J_1}^0\rangle
\langle f_3,  h_{I_3}^0 \otimes  h_{J_3}\rangle$, or one of the other three symmetric cases.

\item We have  $\langle f_1,  h_{I_1}^0 \otimes  h_{J_1}^0\rangle
\langle f_3,  h_{I_3}^0 \otimes  h_{J_3}^0\rangle$.

\end{enumerate}

\textbf{Case 1.}  
We use Lemma \ref{lem:case1} with $f = f_1$, $I = I_1$, $J=J_1$ and $g=f_3$, $Q=I_3$, $R=J_3$. 
Using Lemma \ref{lem:bmobound}, the boundedness property \eqref{eq:bb} and the boundedness of
the operators $A_i(b, \cdot)$, $i = 1,\ldots, 8$, we have that
$$
|\langle [b,U]_1(f_1, f_2), f_3\rangle| \lesssim \max(k_i, v_i) \|f_1\|_{L^p(\R^{n+m})} \|f_2\|_{L^q(\R^{n+m})} \|f_3\|_{L^{r'}(\R^{n+m})}.
$$

\textbf{Case 2.}
This time we use Lemma \ref{lem:case2}. Then we use Lemma \ref{lem:maximalbound}, Lemma \ref{lem:bmobound}, the boundedness property \eqref{eq:bb},
the boundedness of the operators $A_i(b, \cdot)$, $i = 1,\ldots, 8$, the boundedness of the operators $a_i^1(b,\cdot)$, $a_i^2(b,\cdot)$, $i = 1,2$, and Lemma
\ref{lem:bmaxbounds}.
This gives us the desired bound. For example, one calculates like
\begin{align*}
&\sum_{\substack{K \in \calD^n \\ V \in \calD^m}} \sum_{\substack{I_1, I_2, I_3 \in \calD^n \\ I_1^{(k_1)} = I_2^{(k_2)} = I_3^{(k_3)} = K}} 
\sum_{\substack{J_1, J_2, J_3 \in \calD^m \\ J_1^{(v_1)} = J_2^{(v_2)} = J_3^{(v_3)} = V}} |a_{K, V, (I_i), (J_j)}| \\
&\times \big|\bla (\langle b \rangle_{J_1,2}-\langle b \rangle_{I_1 \times J_1}) \langle f_1,h_{J_1} \rangle_2,h^0_{I_1} \bra \big|
|\langle f_2,  \wt h_{I_2} \otimes \wt h_{J_2}\rangle|
|\langle f_3, h_{I_3} \otimes h_{J_3} \rangle| \\
&\lesssim
\sum_{\substack{K \in \calD^n \\ V \in \calD^m}} \sum_{\substack{I_1, I_2, I_3 \in \calD^n \\ I_1^{(k_1)} = I_2^{(k_2)} = I_3^{(k_3)} = K}} 
\sum_{\substack{J_1, J_2, J_3 \in \calD^m \\ J_1^{(v_1)} = J_2^{(v_2)} = J_3^{(v_3)} = V}} |a_{K, V, (I_i), (J_j)}| \\
&\times \langle \varphi_{\calD^m, b}(f_1), h_{I_1}^0 \otimes h_{J_1}\rangle
|\langle f_2,  \wt h_{I_2} \otimes \wt h_{J_2}\rangle|
|\langle f_3, h_{I_3} \otimes h_{J_3} \rangle| \\
&\lesssim \| \varphi_{\calD^m, b}(f_1) \|_{L^p(\R^{n+m})} \|f_2\|_{L^q(\R^{n+m})} \|f_3\|_{L^{r'}(\R^{n+m})} \\
&\lesssim \|f_1\|_{L^p(\R^{n+m})} \|f_2\|_{L^q(\R^{n+m})} \|f_3\|_{L^{r'}(\R^{n+m})}.
\end{align*}

\textbf{Cases 3.-7.}
We operate exactly as above but use Lemma \ref{lem:case3}, or other completely analogous identities (which
are always obtained using the protocol stated in Section \ref{sec:marprod}).
\end{proof}

\section{Quasi--Banach estimates for $\E_{\omega, \omega'}[b,S_{\omega, \omega'}]_1$ via restricted weak type}\label{sec:quasiviarest}
In this section we will prove:
\begin{thm}\label{thm:com1ofmodelQuasiBanach}
Let $\|b\|_{\bmo(\R^{n+m})}  = 1$, and let $1 < p, q \le \infty$ and $1/2 < r < \infty$ satisfy $1/p+1/q = 1/r$.
Suppose $S_{\omega, \omega'} := S^{v}_{k, \mathcal{D}^n_{\omega},\mathcal{D}^m_{\omega'}}$
is a bilinear bi-parameter shift of complexity $(k,v)$ defined using the dyadic grids $\mathcal{D}^n_{\omega}$ and $\mathcal{D}^m_{\omega'}$.
Then we have
$$
 \|\mathbb{E}_{\omega, \omega'}[b,S_{\omega, \omega'}]_1(f_1, f_2)\|_{L^r(\R^{n+m})} 
 \lesssim (1+\max(k_i, v_i)) \|f_1\|_{L^p(\R^{n+m})} \|f_2\|_{L^q(\R^{n+m})}.
$$
\end{thm}

The case when $r>1$ in Theorem \ref{thm:com1ofmodelQuasiBanach} is easy, since we already know 
the Banach range boundedness of commutators of shifts. 
Indeed, if $f_1 \in L^p(\R^{n+m})$, $f_2 \in L^q(\R^{n+m})$ and $f_3 \in L^{r'}(\R^{n+m})$, then
\begin{equation*}
\begin{split}
| \langle \mathbb{E}_{\omega, \omega'}[b,S_{\omega, \omega'}]_1(f_1, f_2),f_3 \rangle |
& \le \E_{\omega,\omega'} | \langle [b,S_{\omega, \omega'}]_1(f_1, f_2),f_3 \rangle | \\
&\lesssim (1+\max(k_i, v_i)) \|f_1\|_{L^p(\R^{n+m})} \|f_2\|_{L^q(\R^{n+m})}
\|f_3\|_{L^{r'}(\R^{n+m})}. 
\end{split}
\end{equation*}

The main task is to prove a restricted weak type estimate, which combined with the Banach range boundedness
implies Theorem \ref{thm:com1ofmodelQuasiBanach} via interpolation. We will show that given $p,q \in (1, \infty)$ and $r \in (1/2,1)$ satisfying $1/p+1/q=1/r$, 
$f_1 \in L^p(\R^{n+m})$, $f_2 \in L^q(\R^{n+m})$ and a set $E \subset \R^{n+m}$ with $0 < |E| < \infty$,
there exists a subset $E' \subset E$ such that $|E'| \ge |E|/2$ and such that for all functions $f_3$
satisfying $|f_3| \le 1_{E'}$ there holds
\begin{equation}\label{eq:ResAveShift}
\begin{split}
| \langle \mathbb{E}_{\omega, \omega'}&[b,S_{\omega, \omega'}]_1(f_1, f_2),f_3 \rangle | \\
&\lesssim (1+\max(k_i, v_i)) \|f_1\|_{L^p(\R^{n+m})} \|f_2\|_{L^q(\R^{n+m})}|E|^{1/r'}.
\end{split}
\end{equation}

To prove \eqref{eq:ResAveShift} we consider the different types of shifts separately
and split the commutators using the identities from Section \ref{sec:KeyIdentities}.
We choose one particular term and show the proof with it in all detail. This
pretty well describes the steps needed for the other terms also, and we shall comment on this in the end.

We will denote the coefficients related to the shift $S_{\omega, \omega'}$ by $a^{\omega,\omega'}_{K,V,(I_i),(J_i)}$.
The shifts we consider here are of the form
\begin{equation}\label{eq:ResAveShiftMixed}
\begin{split}
\langle S_{\omega, \omega'}(f_1,f_2),f_3 \rangle
= \sum_{\substack{K \in \calD^n_\omega \\ V \in \calD^m_{\omega'}}} 
&\sum_{\substack{I_1, I_2, I_3 \in \calD^n_\omega \\ I_i^{(k_i)} = K}} 
\sum_{\substack{J_1, J_2, J_3 \in \calD^m_{\omega'} \\ J_i^{(v_i)}  = V}} a^{\omega,\omega'}_{K, V, (I_i), (J_j)} \\
&\times \langle f_1,  h_{I_1}^0 \otimes   h_{J_1}\rangle \langle f_2,  h_{I_2} \otimes  h_{J_2}\rangle  
\langle f_3,   h_{I_3} \otimes  h_{J_3}^0 \rangle.
\end{split}
\end{equation}
To the commutator of this we apply Lemma \ref{lem:case3}. One of the resulting terms, which is the term that we handle in detail,
is considered in the next lemma.

\begin{lem}\label{lem:ResAveShiftEx}
Let $\| b \|_{\bmo(\R^{n+m})}=1$ and let $p,q \in (1, \infty)$ and $r \in (1/2,1)$ satisfy $1/p+1/q=1/r$. 
Suppose $f_1 \in L^p(\R^{n+m})$, $f_2 \in L^q(\R^{n+m})$ and  $E \subset \R^{n+m}$ with $0 < |E| < \infty$.
Then there exists a subset $E' \subset E$ with $|E'| \ge \frac{99}{100} |E|$  so that for all functions $f_3$
satisfying $|f_3| \le 1_{E'}$ there holds
\begin{equation}\label{eq:ResAveShiftExEst}
\begin{split}
\Big|\E_{\omega,\omega'}\sum_{\substack{K \in \calD^n_0 \\ V \in \calD^m_0}} 
 \Lambda^{\omega,\omega'}_{K \times V} (f_1,f_2,f_3 )\Big| 
 \lesssim \|f_1\|_{L^p(\R^{n+m})} \|f_2\|_{L^q(\R^{n+m})}|E|^{1/r'},
\end{split}
\end{equation}
where
\begin{equation*}
\begin{split}
\Lambda^{\omega,\omega'}_{K \times V} (f_1,f_2,f_3 )
= &\sum_{\substack{I_1, I_2, I_3 \in \calD^n_\omega \\ I_i^{(k_i)} = K+\omega}} 
\sum_{\substack{J_1, J_2, J_3 \in \calD^m_{\omega'} \\ J_i^{(v_i)}  = V+\omega'}} 
a^{\omega,\omega'}_{K + \omega, V+\omega', (I_i), (J_j)}  \\
&\times  \langle f_1,  h_{I_1}^0 \otimes   h_{J_1}\rangle \langle f_2,  h_{I_2} \otimes  h_{J_2}\rangle  
\bla (\langle b \rangle_{I_3,1}-\langle b \rangle_{I_3 \times J_3}) \langle  f_3, h_{I_3} \rangle_1,    h_{J_3}^0 \bra.
\end{split}
\end{equation*}
\end{lem}

For the proof of Lemma \ref{lem:ResAveShiftEx} we record  the boundedness of certain deterministic square functions.
Let $i,j \in \Z$,  $i, j \ge 0$. Suppose that we have a family of operators $U=\{U_ {\omega,\omega'}\}_{\omega,\omega'}$ such that for all $\omega,\omega'$ there holds
$$
\| U_{\omega,\omega'} f \|_{L^2(w)} \le C([w]_{A_2(\R^n \times \R^m)})\| f \|_{L^2(w)}, \quad f \in L^2(w),
$$
for all $w \in A_2(\R^n \times \R^m)$. Recall that $\calD_0 = \calD^n_0 \times \calD^m_0$, where $\calD^n_0$ and $\calD^m_0$ are
the standard dyadic grids of $\R^n$ and $\R^m$ respectively.
Define
$$
S_{U}^{i,j} f
=\Big(\sum_{K \times V \in \calD_0} \E_{\omega,\omega'} (M \Delta^{i,j}_{(K+\omega) \times (V+\omega')} U_{\omega,\omega'}f)^2 \Big)^{1/2}.
$$
Given a similar $U=\{U_ {\omega}\}_{\omega}$ we set
$$
S^1_{i,U}=\Big(\sum_{K \in \calD_0^n} \E_{\omega} (M \Delta^{1}_{K+\omega,i} U_{\omega} f)^2 \Big)^{1/2}
$$
and given $U=\{U_ {\omega'}\}_{\omega'}$ we set
$$
S^2_{j,U}=\Big(\sum_{V \in \calD_0^m} \E_{\omega'} (M \Delta^{2}_{V+\omega',j} U_{\omega'} f)^2 \Big)^{1/2}.
$$
We write $S^{i,j}$, $S^1_{i}$ and $S^2_j$ if there is no $U$ present.

\begin{lem}\label{lem:DetSquareShift}
For all $p \in (1, \infty)$ and $w \in A_p(\R^n\times\R^m)$ there holds
$$
\| S^{i,j}_Uf\|_{L^p(w)}+\| S^{1}_{i,U}f\|_{L^p(w)}+\| S^{2}_{j,U}f\|_{L^p(w)}
\le C([w]_{A_p(\R^n\times\R^m)})  \| f\|_{L^p(w)}.
$$
\end{lem}

\begin{proof}
We show the estimate for $S^{i,j}_U$. The other two are very similar. 
Suppose $w \in A_2(\R^n \times \R^m)$ and $f \in L^2(w)$. Then
\begin{align*}
\| S^{i,j}_Uf\|_{L^2(w)}^2
&= \E_{\omega,\omega'} \sum_{K \times V \in \calD_0} \| M \Delta^{i,j}_{(K+\omega) \times (V+\omega')} U_{\omega,\omega'}f \|_{L^2(w)}^2 \\
&\le C([w]_{A_2(\R^n \times \R^m)}) \E_{\omega,\omega'} \|  U_{\omega,\omega'}f \|_{L^2(w)}^2
\le C([w]_{A_2(\R^n \times \R^m)}) \| f \|_{L^2(w)}^2,
\end{align*}
where in the second step we used the weighted boundedness of the strong maximal function and the usual rectangular dyadic square function.
The claim for $p \in (1,\infty)$ follows from extrapolation.
\end{proof}
The following estimates contain the standard estimates for shifts, and are the reason why various square functions arise naturally.
For $R = K \times V \in \calD_0$ we denote $R_{\omega, \omega'} = (K+\omega) \times (V+\omega') \subset 3K \times 3V = 3R$. Then 
using the normalisation of the constants $a^{\omega,\omega'}_{K + \omega, V+\omega', (I_i), (J_j)}$ and adding martingale differences
using cancellative Haar functions we have 
\begin{align*}
&\sum_{\substack{I_1, I_2, I_3 \in \calD^n_\omega \\ I_i^{(k_i)} = K+\omega}} 
\sum_{\substack{J_1, J_2, J_3 \in \calD^m_{\omega'} \\ J_i^{(v_i)}  = V+\omega'}} 
|a^{\omega,\omega'}_{K + \omega, V+\omega', (I_i), (J_j)}| \\
 & \hspace{2cm} \times |\langle U_{1, \omega'} f_1,  h_{I_1}^0 \otimes   h_{J_1}\rangle| |\langle U_{2, \omega, \omega'} f_2,  h_{I_2} \otimes  h_{J_2}\rangle|  
|\langle U_{3, \omega} f_3, h_{I_3} \otimes h_{J_3}^0 \rangle| \\
&\le \langle |\Delta^{2}_{V+\omega',v_1}U_{1, \omega'} f_1 |\rangle_{R_{\omega, \omega'}}
\langle |\Delta^{k_2,v_2}_{K+\omega, V + \omega'}U_{2, \omega, \omega'} f_2 |\rangle_{R_{\omega, \omega'}}\langle |\Delta^{1}_{K+\omega,k_3}U_{3, \omega} f_3 |\rangle_{R_{\omega, \omega'}} |R| \\
&\lesssim \int \langle |\Delta^{2}_{V+\omega',v_1}U_{1, \omega'} f_1 |\rangle_{3R}
\langle |\Delta^{k_2,v_2}_{K+\omega, V + \omega'}U_{2, \omega, \omega'} f_2 |\rangle_{3R}\langle |\Delta^{1}_{K+\omega,k_3}U_{3, \omega} f_3 |\rangle_{3R}1_R \\
&\le \int M \Delta^{2}_{V+\omega',v_1}U_{1, \omega'} f_1 \cdot M\Delta^{k_2,v_2}_{K+\omega, V + \omega'}U_{2, \omega, \omega'} f_2 \cdot M \Delta^{1}_{K+\omega,k_3}U_{3, \omega} f_3,
\end{align*}
and so furthermore
\begin{align*}
&\E_{\omega, \omega'} \sum_{\substack{K \in \calD^n_0 \\ V \in \calD^m_0}}\sum_{\substack{I_1, I_2, I_3 \in \calD^n_\omega \\ I_i^{(k_i)} = K+\omega}} 
\sum_{\substack{J_1, J_2, J_3 \in \calD^m_{\omega'} \\ J_i^{(v_i)}  = V+\omega'}} 
|a^{\omega,\omega'}_{K + \omega, V+\omega', (I_i), (J_j)}| \\
 & \hspace{2cm} \times |\langle U_{1, \omega'} f_1,  h_{I_1}^0 \otimes   h_{J_1}\rangle| |\langle U_{2, \omega, \omega'} f_2,  h_{I_2} \otimes  h_{J_2}\rangle|  
|\langle U_{3, \omega} f_3, h_{I_3} \otimes h_{J_3}^0 \rangle| \\
&\lesssim \int \sum_{\substack{K \in \calD^n_0 \\ V \in \calD^m_0}} \E_{\omega, \omega'} M \Delta^{2}_{V+\omega',v_1}U_{1, \omega'} f_1 \cdot M\Delta^{k_2,v_2}_{K+\omega, V + \omega'}U_{2, \omega, \omega'} f_2 \cdot M \Delta^{1}_{K+\omega,k_3}U_{3, \omega} f_3 \\
&\le \int \sum_{\substack{K \in \calD^n_0 \\ V \in \calD^m_0}} (\E_{\omega, \omega'} (M \Delta^{2}_{V+\omega',v_1}U_{1, \omega'} f_1)^2 (M \Delta^{1}_{K+\omega,k_3}U_{3, \omega} f_3)^2 )^{1/2} \\
& \hspace{6cm}\times(\E_{\omega, \omega'} (M\Delta^{k_2,v_2}_{K+\omega, V + \omega'}U_{2, \omega, \omega'} f_2)^2 )^{1/2} \\
&\le \int \Big( \sum_{\substack{K \in \calD^n_0 \\ V \in \calD^m_0}} \E_{\omega, \omega'} (M \Delta^{2}_{V+\omega',v_1}U_{1, \omega'} f_1)^2 (M \Delta^{1}_{K+\omega,k_3}U_{3, \omega} f_3)^2 \Big)^{1/2} \\
&\hspace{6cm} \times  \Big( \sum_{\substack{K \in \calD^n_0 \\ V \in \calD^m_0}} \E_{\omega, \omega'} (M\Delta^{k_2,v_2}_{K+\omega, V + \omega'}U_{2, \omega, \omega'} f_2)^2 \Big)^{1/2} \\
&= \int \Big( \sum_{V \in \calD^m_0} \E_{\omega'} (M \Delta^{2}_{V+\omega',v_1}U_{1, \omega'} f_1)^2 \Big)^{1/2}
\Big( \sum_{\substack{K \in \calD^n_0 \\ V \in \calD^m_0}} \E_{\omega, \omega'} (M\Delta^{k_2,v_2}_{K+\omega, V + \omega'}U_{2, \omega, \omega'} f_2)^2 \Big)^{1/2}  \\
&\hspace{6cm} \times \Big( \sum_{K \in \calD^n_0} \E_{\omega} (M \Delta^{1}_{K+\omega,k_3}U_{3, \omega} f_3)^2 \Big)^{1/2}.
\end{align*}
Estimates in the above spirit are used repeatedly below.

We are now ready to prove Lemma \ref{lem:ResAveShiftEx}.

\begin{proof}[Proof of Lemma \ref{lem:ResAveShiftEx}]
Because the square functions $S^2_{v_1}$ and $S^{k_2,v_2}$ are bounded, it is enough to assume that
$$
\| S^2_{v_1}f_1 S^{k_2,v_2}f_2 \|_{L^r}=1,
$$
and show that there exists a set $E' \subset E$ with $|E'| \ge \frac{99}{100} |E|$ so that the left hand side of \eqref{eq:ResAveShiftExEst}
is dominated by $|E|^{1/r'}$.

Define
$$
\Omega_u =\{S^2_{v_1}f_1 S^{k_2,v_2}f_2 > C_0 2^{-u}|E|^{-1/r}\}, \quad u \ge 0,
$$
and 
$$
\wt \Omega_u = \{M1_{\Omega_u}> c_1\},
$$
where $c_1>0$ is a small enough dimensional constant. Then we can
choose $C_0=C_0(c_1)$ so large that the set $E':=E \setminus \wt \Omega_0$ satisfies $|E'|  \ge \frac{99}{100} |E|$.
Then we define the collections
$$
\widehat \calR_u =\Big\{ R \in \calD_0 \colon |R \cap \Omega_u | \ge \frac{|R|}{2}\Big\},
$$
where $\calD_0 = \calD^n_0 \times \calD^m_0$.

Below we denote $(K+\omega) \times (V+ \omega') = K \times V+ (\omega, \omega')$.
Let us list some properties of the collections $\widehat \calR_u$. Fix now some function $f_3$ such that 
$|f_3| \le 1_{E'}$. 
Suppose $K \times V \in \calD_0$ is such that
$$
\E_{\omega,\omega'}  \Lambda^{\omega,\omega'}_{K \times V} (f_1,f_2,f_3)  \ne 0.
$$
From this it can be concluded that
\begin{equation*}
\begin{split}
0 < \E_{\omega,\omega'} 
&\langle |\Delta^{2}_{V+\omega',v_1}f_1 |\rangle_{K\times V+(\omega,\omega')}
\langle |\Delta^{k_2,v_2}_{K\times V+(\omega,\omega')}f_2 |\rangle_{K\times V+(\omega,\omega')} \\
& \lesssim \inf_{K \times V} S^{2}_{v_1} f_1 S^{k_2,v_2} f_2.
\end{split}
\end{equation*}
Therefore, every relevant $K \times V \in \calD_0$ that appears in the summation in \eqref{eq:ResAveShiftExEst} is contained in  $\Omega_u$ for some $u$,
and therefore belongs to $\widehat \calR_u$.
Also, if $K \times V \in \widehat \calR_u$, then for all $(\omega,\omega')$ we have
$K\times V+(\omega,\omega') \subset 3K \times 3V \subset \wt \Omega_u$. Here we used the fact that 
$c_1$ is small enough. Thus, if $K \times V \in \widehat \calR_0$, then for all $(\omega,\omega')$ we have $(K+\omega) \times (V +\omega') \cap E'=\emptyset$.
We want to say that this implies that for all $(\omega, \omega')$ we have $\Lambda^{\omega,\omega'}_{K \times V} (f_1,f_2,f_3)=0$, and
so all relevant $K \times V$ satisfy $K \times V \not \in \widehat \calR_0$. The claim in the previous sentence is based on the following localisation property
\begin{equation}\label{eq:localisation}
\bla (\langle b \rangle_{I_3,1}-\langle b \rangle_{I_3 \times J_3}) \langle  f_3, h_{I_3} \rangle_1,    h_{J_3}^0 \bra
=\bla (\langle b \rangle_{I_3,1}-\langle b \rangle_{I_3 \times J_3}) \langle  1_{I_3 \times J_3}f_3, h_{I_3} \rangle_1,    h_{J_3}^0 \bra,
\end{equation}
and the fact that $|f_3| \le 1_{E'}$. The localisation will also be used below to see that we can replace $f_3$ by $f_31_F$ for
any set $F \supset I_3 \times J_3$.

Define the collections $\calR_u= \widehat \calR_u \setminus \widehat \calR_{u-1}$, where $u \ge 1$.
We have demonstrated that every relevant $K \times V \in \calD_0$ appearing in the summation in \eqref{eq:ResAveShiftExEst}  belongs to exactly one of these collections. Therefore, we have
$$
\E_{\omega,\omega'}\sum_{\substack{K \in \calD^n_0 \\ V \in \calD^m_0}} 
 \Lambda^{\omega,\omega'}_{K \times V} (f_1,f_2,f_3 )
=\sum_{u=1}^\infty \sum_{K \times V\in \calR_u} 
\E_{\omega,\omega'} \Lambda^{\omega,\omega'}_{K \times V} (f_1,f_2,f_3 ).
$$
We now fix one $u$ and estimate the corresponding term.
Using now that $I_3 \times J_3 \subset (K+\omega) \times (V+\omega') \subset \wt \Omega_u$, since
$K \times V \in \calR_u \subset \widehat \calR_u$, we may replace $f_3$ with $1_{\wt \Omega_u}f_3$.

Let us write $\varphi_{\omega, b} = \varphi_{\calD^n_{\omega}, b}$.
Suppose $K \times V \in \calR_u$. Using Lemma \ref{lem:maximalbound} we have
\begin{equation*}
\begin{split}
| \Lambda^{\omega,\omega'}_{K \times V} (f_1,f_2,1_{\wt \Omega_u}f_3 )| 
  &\le \langle |\Delta^{2}_{V+\omega',v_1}f_1 |\rangle_{K\times V+(\omega,\omega')} 
  \langle |\Delta^{k_2,v_2}_{K\times V+(\omega,\omega')}f_2 |\rangle_{K\times V+(\omega,\omega')} \\
& \times \langle |\Delta^{1}_{K+\omega,k_3}\varphi_{\omega, b} (1_{\wt \Omega_u}f_3) |\rangle_{K\times V+(\omega,\omega')} |K\times V|.
\end{split}
\end{equation*}
Since $K \times V \in \calR_u$, we also have that $|K \times V| \lesssim \int_{\wt \Omega_u \setminus \Omega_{u-1}} 1_{K \times V}$.
Combining these there holds that
\begin{equation*}
\begin{split}
 \E_{\omega,\omega'} &| \Lambda^{\omega,\omega'}_{K \times V} (f_1,f_2,f_3 )| \\
&  \lesssim \int_{\wt \Omega_u \setminus \Omega_{u-1}} 
 \E_{\omega,\omega'} M (\Delta^2_{V+\omega',v_1} f_1) M(\Delta^{k_2,v_2}_{K \times V+(\omega,\omega')} f_2)
 M (\Delta^1_{K+\omega,k_3} \varphi_{\omega, b} (1_{\wt \Omega_u}f_3)).
\end{split}
\end{equation*}
Recalling that $\E_{\omega,\omega'}= \E_\omega \E_{\omega'}$, we notice that the last integrand is pointwise dominated by
$$
\big(\E_{\omega'} (M \Delta^2_{V+\omega',v_1} f_1)^2\big)^{\frac{1}{2}}
\big(\E_{\omega,\omega'} (M\Delta^{k_2,v_2}_{K \times V+(\omega,\omega')} f_2)^2\big)^{\frac{1}{2}}
\big(\E_{\omega} (M \Delta^1_{K+\omega, k_3} \varphi_{\omega, b}(1_{\wt \Omega_u} f_3)^2\big)^{\frac{1}{2}}.
$$
Using this, we finally have that
\begin{equation}\label{eq:SquaresAppear}
\sum_{K \times V\in \calR_u} \E_{\omega,\omega'}
| \Lambda^{\omega,\omega'}_{K \times V} (f_1,f_2,f_3 )|
\lesssim \int_{\wt \Omega_u \setminus \Omega_{u-1}}
S^2_{v_1} f_1 S^{k_2,v_2} f_2 S^1_{k_3, \varphi_b^1} (1_{\wt \Omega_u}f_3),
\end{equation}
where $S^1_{k_3, \varphi_b^1}$ is the square function formed with the family $\varphi_b^1 := \{\varphi_{\omega, b}\}_{\omega}$ 
as in Lemma \ref{lem:DetSquareShift}.

If $x  \not \in \Omega_{u-1}$, then by definition $S^2_{v_1} f_1(x) S^{k_2,v_2} f_2(x) \lesssim 2^{-u}|E|^{-1/r}$.
Thus, the right hand side of \eqref{eq:SquaresAppear} is dominated by
\begin{equation*}
2^{-u}|E|^{-1/r} | \wt \Omega_{u} |^{1/2}
\| S^1_{k_3, \varphi_b^1} (1_{\wt \Omega_u}f_3)\|_{L^2}
\lesssim 2^{-u(1-r)}|E|^{1-1/r}.
\end{equation*}
The boundedness of $S^1_{k_3, \varphi_b^1}$ is based on Lemma \ref{lem:bmaxbounds} and Lemma \ref{lem:DetSquareShift}.
The last estimate can be summed over $u$, since $r < 1$. This concludes the proof.
\end{proof}

Based on the proof of Lemma \ref{lem:ResAveShiftEx}, 
we shall now comment on the other terms that arise when we apply Lemma \ref{lem:case3}
to the commutator of a shift of the form \eqref{eq:ResAveShiftMixed}. 
First, we consider the variant of Lemma \ref{lem:ResAveShiftEx},
where in the definition of $\Lambda^{\omega,\omega'}_{K \times V}(f_1,f_2,f_3)$  we have replaced the pairing
$\bla (\langle b \rangle_{I_3,1}-\langle b \rangle_{I_3 \times J_3}) \langle  f_3, h_{I_3} \rangle_1,    h_{J_3}^0 \bra$
with $\langle a^1_{i,\omega}(b,f_3), h_{I_3} \otimes h_{J_3}^0 \rangle$, where $i=1,2$.
In this case we construct the sets $\Omega_u$, $\wt \Omega_u$ and the collections $ \calR_u$ precisely as
in Lemma \ref{lem:ResAveShiftEx}. Again $E'=E \setminus \wt \Omega_0$, and we assume that
$|f_3| \le 1_{E'}$.
 From the definition of the operators $a^1_{i,\omega}$ one sees that 
if $I_3 \times J_3 \in \calD_{\omega}^n \times \calD^m_{\omega'}$, we have a similar localisation property as in
\eqref{eq:localisation}, namely
$$
\langle a^1_{i,\omega}(b,f_3), h_{I_3} \otimes h_{J_3}^0 \rangle
=\langle a^1_{i,\omega}(b,1_{I_3 \times J_3}f_3), h_{I_3} \otimes h_{J_3}^0 \rangle.
$$

We can again organise the sum over $K \times V \in \calD_0$ as
\begin{equation*}
\E_{\omega,\omega'}\sum_{K \times V \in \calD_0} \Lambda^{\omega,\omega'}_{K \times V}(f_1,f_2,f_3)
=\sum_{u=1}^\infty \sum_{K \times V \in \calR_u} 
\E_{\omega,\omega'}\Lambda^{\omega,\omega'}_{K \times V}(f_1,f_2,f_3).
\end{equation*}
Let $a^1_{i,b}$ denote the family of operators $\{a^1_{i,\omega}(b,\cdot)\}_{\omega}$. For a fixed $u$, we have corresponding
to \eqref{eq:SquaresAppear} that
\begin{equation*}
\sum_{K \times V\in \calR_u} \E_{\omega,\omega'}
| \Lambda^{\omega,\omega'}_{K \times V} (f_1,f_2,f_3 )|
\lesssim \int_{\wt \Omega_u \setminus \Omega_{u-1}}
S^2_{v_1} f_1 S^{k_2,v_2} f_2 S^1_{k_3, a^1_{i,b}} (1_{\wt \Omega_u}f_3).
\end{equation*}
The boundedness of $S^1_{k_3, a^1_{i,b}}$ is based on \eqref{eq:wforsmallA} and Lemma \ref{lem:DetSquareShift}.
From here the proof can be concluded as in Lemma \ref{lem:ResAveShiftEx}.

Next, we consider the case where $\Lambda^{\omega,\omega'}_{K \times V}$ is defined by
$$
\sum_{\substack{I_1, I_2, I_3 \in \calD^n_\omega \\ I_i^{(k_i)} = K+\omega}} 
\sum_{\substack{J_1, J_2, J_3 \in \calD^m_{\omega'} \\ J_i^{(v_i)}  = V+\omega'}} 
a^{\omega,\omega'}_{K+\omega, V+\omega', (I_i), (J_j)}  \\
  \langle a^2_{i,\omega'}(b, f_1),  h_{I_1}^0 \otimes   h_{J_1}\rangle \langle f_2,  h_{I_2} \otimes  h_{J_2}\rangle  
\langle  f_3, h_{I_3} \otimes  h_{J_3}^0 \rangle,
$$
where $i=1,2$. This time we define
$$
\Omega_u = \{ S^2_{v_1,a^2_{i,b}}f_1 S^{k_2,v_2}f_2 > C_0 2^{-u}|E|^{-1/r}\}.
$$
Based on these, we construct the sets $\wt \Omega_u$ and the collections $\calR_u$ as before. This time the localisation property
is clear, as $f_3$ is free. For a fixed $u$, we have corresponding
to \eqref{eq:SquaresAppear} that
\begin{equation*}
\sum_{K \times V\in \calR_u} \E_{\omega,\omega'}
| \Lambda^{\omega,\omega'}_{K \times V} (f_1,f_2,f_3 )|
\lesssim \int_{\wt \Omega_u \setminus \Omega_{u-1}}
S^2_{v_1,a^2_{i,b}}f_1 S^{k_2,v_2}f_2 S^1_{k_3} (1_{\wt \Omega_u}f_3),
\end{equation*}
and the proof can be concluded analogously. 

The remaining cases are: we have $\bla (\langle b \rangle_{J_1,2}-\langle b \rangle_{I_1 \times J_1}) \langle  f_1, h_{J_1} \rangle_2,    h_{I_1}^0 \bra$, or
we have the factor $\langle b \rangle_{I_3 \times J_3}- \langle b \rangle_{I_1 \times J_1}$ at the front. These can be done similarly, the last one
being easiest due to Lemma \ref{lem:bmobound}. We have now proved \eqref{eq:ResAveShift} for the shifts of the type \eqref{eq:ResAveShiftMixed}.

\subsection*{Shifts of other type}
Let us briefly comment on commutators of shifts that are of different type than above.
Depending on the shift, the identities from Section \ref{sec:KeyIdentities} give various terms.
These are all handled similarly as above, the main difference being in the construction of the sets $\Omega_u$ and in
the use of  different combinations of square functions and maximal functions.
We give a few indications of the required modifications.

We did not encounter the $A_{i, \omega, \omega'}(b, \cdot)$ operators above, so we comment on a few cases which entail them.
Suppose we are dealing with terms of the form
$$
a^{\omega,\omega'}_{K+\omega, V+\omega', (I_i), (J_j)}
\langle A_{i, \omega, \omega'}(b, f_1),  h_{I_1} \otimes   h_{J_1}\rangle \langle f_2,  h_{I_2}^0 \otimes  h_{J_2}^0\rangle  
\langle  f_3, h_{I_3} \otimes  h_{J_3} \rangle,
$$
where $i=1,\dots,8$. Let $A_{i,b}=\{A_{i,\omega,\omega'}(b, \cdot) \}_{\omega,\omega'}$ and let $S^{k_1,v_1}_{A_{i,b}}$ be the related square function.
This time one defines
$$
\Omega_u=\{S^{k_1,v_1}_{A_{i,b}}f_1 Mf_2 > C_0 2^{-u} |E|^{-1/r}\}.
$$
The boundedness of $S^{k_1,v_1}_{A_{i,b}}$ is based on \eqref{eq:wforlargeA1}, \eqref{eq:wforlargeA2} and Lemma \ref{lem:DetSquareShift}.
The proof proceeds as previously.

Related to terms 
$$
a^{\omega,\omega'}_{K+\omega, V+\omega', (I_i), (J_j)}
\langle f_1,  h_{I_1} \otimes   h_{J_1}\rangle \langle f_2,  h_{I_2}^0 \otimes  h_{J_2}^0\rangle  
\langle  A_{i, \omega, \omega'}(b, f_3), h_{I_3} \otimes  h_{J_3} \rangle,
$$
one sets
$$
\Omega_u=\{S^{k_1,v_1}f_1 Mf_2 > C_0 2^{-u} |E|^{-1/r}\}.
$$
Corresponding to the key localisation property \eqref{eq:localisation}, the operators $A_{i,\omega,\omega'}(b, \cdot)$ 
satisfy that if $I_3 \times J_3 \in \calD^n_\omega \times \calD^m_{\omega'}$ then
$$
\langle  A_{i, \omega, \omega'}(b, f_3), h_{I_3} \otimes  h_{J_3} \rangle
=\langle  A_{i, \omega, \omega'}(b, 1_{I_3 \times J_3}f_3), h_{I_3} \otimes  h_{J_3} \rangle.
$$
In the proof one uses related to $f_3$ the square function $S^{k_3,v_3}_{A_{i,b}}f_3$.

Finally, terms of the form
$$
a^{\omega,\omega'}_{K+\omega, V+\omega', (I_i), (J_j)}
\langle f_1,  h_{I_1} \otimes   h_{J_1}\rangle \langle f_2,  h_{I_2} \otimes  h_{J_2}\rangle  
\langle  (b-\langle b \rangle_{I_3 \times J_3})f_3, h_{I_3}^0 \otimes  h_{J_3}^0 \rangle
$$
are also easy to handle via Lemma \ref{lem:maximalbound}.

\subsection{Concluding the proof of Theorem \ref{thm:com1ofmodelQuasiBanach}}
Having now proved \eqref{eq:ResAveShiftMixed} for all shift types, it only remains to interpolate to get Theorem \ref{thm:com1ofmodelQuasiBanach}.
Let now $1/p + 1/q = 1/r$, $1 < p, q < \infty$, and $S_{\omega, \omega'}$ be a shift of any type.
At this point we know that for $r > 1$ we have
\begin{equation}\label{eq:BRange}
\|\E_{\omega,\omega'} [b,S_{\omega, \omega'}]_1(f_1, f_2)\|_{L^r(\R^{n+m})} 
\lesssim (1+\max(k_i, v_i))\|f_1\|_{L^p(\R^{n+m})} \|f_2\|_{L^q(\R^{n+m})}
\end{equation}
and for $r < 1$ we have
\begin{equation}\label{eq:Weak}
 \|\E_{\omega,\omega'} [b,S_{\omega, \omega'}]_1(f_1, f_2)\|_{L^{r, \infty}(\R^{n+m})} 
\lesssim (1+\max(k_i, v_i))\|f_1\|_{L^p(\R^{n+m})} \|f_2\|_{L^q(\R^{n+m})}.
\end{equation}
Notice that for $r = 1$ we may easily get that
if $0 < |E_i| < \infty$, $i = 1,2,3$, there exists $E_3' \subset E_3$ so that $|E_3'| \ge |E_3|/2$, and so that
for all $|f_1| \le 1_{E_1}$, $|f_2| \le 1_{E_2}$ and $|f_3| \le 1_{E_3'}$ we have
$$
|\langle \E_{\omega,\omega'} [b,S_{\omega, \omega'}]_1(f_1, f_2), f_3\rangle| \lesssim (1+\max(k_i, v_i)) |E_1|^{1/p}|E_2|^{1/p'}. 
$$
This follows by taking convex combinations of our existing estimates \eqref{eq:BRange}, \eqref{eq:Weak}.
Then use e.g. Theorem 3.8 in Thiele's book \cite{Th:Book} to update all of our estimates that are either
weak type (if $r < 1$) or restricted weak type (if $r=1$) into strong type bounds. Finally, notice that the cases $p= \infty$ or $q = \infty$ can now
be obtained by duality. Indeed, let $p = \infty$ and $r = q \in (1,\infty)$. Then we have
$$
|\langle \E_{\omega,\omega'} [b, S_{\omega, \omega'}]_1(f_1, f_2), f_3\rangle| 
= |\langle \E_{\omega,\omega'} [b, S^{1*}_{\omega, \omega'}]_1(f_3, f_2), f_1\rangle|,
$$
where we used that $[b, S_{\omega, \omega'}]_1^{1*} = -[b, S_{\omega, \omega'}^{1*}]_1$. It remains to use the at this point already known
bound
$$
\| \E_{\omega,\omega'} [b, S^{1*}_{\omega, \omega'}]_1(f_3, f_2)\|_{L^1(\R^{n+m})} \lesssim (1+\max(k_i, v_i)) \|f_3\|_{L^{q'}(\R^{n+m})}\|f_2\|_{L^{q}(\R^{n+m})}.
$$
We have proved Theorem \ref{thm:com1ofmodelQuasiBanach}.

\section{Iterated commutators}\label{sec:iterated}
For the iterated commutators $[b_2, [b_1, T]_1]_2$ (or $[b_2, [b_1, T]_1]_1$)
we need weighted bounds for the linear commutators $[b_2, A_i(b_1, \cdot)]$ and
$[b_2, a_i^1(b_1, \cdot)]$. The need arises similarly as previously when we needed weighted bounds for $A_i(b_1, \cdot)$
and $a_i^1(b_1, \cdot)$ (stated in Section \ref{sec:marprod}) when considering $[b_1, T]_1$. Again, the \emph{weighted}
versions are only needed for the boundedness of some square functions as in Lemma \ref{lem:DetSquareShift}, which
are needed in the quasi--Banach estimates. As these are linear estimates, some of them were already
considered in \cite{HPW} -- namely, for $i = 1,2,3,4$ they even proved Bloom type two-weight estimates for $[b_2, A_i(b_1, \cdot)]$, when
$b_2 \in \bmo(\R^n \times \R^m)$ and $b_1 \in \BMO_{\textup{prod}}(\R^{n+m})$. However, we need the one-weight
versions also for $i = 5, 6,7,8$, and the proofs are quite straightforward with the our by now familiar method. As we do not have any use
for Bloom type estimates, we content here by giving a quick proof of the one-weight result.
\begin{lem}\label{lem:weightedoneparcommutator}
Let $\|b_1\|_{\bmo(\R^{n+m})} = \|b_2\|_{\bmo(\R^{n+m})} = 1$, $1 < p < \infty$ and $w \in A_p(\R^n \times \R^m)$.
Then for $i = 1, \ldots, 8$ we have
$$
\| [b_2, A_i(b_1, f)] \|_{L^p(w)} \le C([w]_{A_p(\R^n \times \R^m)}) \|f\|_{L^p(w)}
$$
and for $i = 1,2$ we have
$$
\| [b_2, a_i^1(b_1, f)] \|_{L^p(w)} \le C([w]_{A_p(\R^n \times \R^m)}) \|f\|_{L^p(w)}.
$$
\end{lem}
\begin{proof}
We only prove
$$
\| [b_2, A_5(b_1, f)] \|_{L^p(w)} \le C([w]_{A_p(\R^n \times \R^m)}) \|f\|_{L^p(w)},
$$
the rest of the cases being similar. Denote $\lambda_{I,J}^{b_1} = \big\langle b_1, \frac{1_I}{|I|} \otimes h_J\big\rangle$. The estimate
\begin{equation}\label{eq:A5weighted}
\sum_{I,J} |\lambda_{I,J}^{b_1}| |\langle f_1, h_I \otimes h_J\rangle|  \langle |\langle f_2, h_I \rangle_1| \rangle_J
\le C([w]_{A_p(\R^n \times \R^m)}) \|f_1\|_{L^p(w)} \|f_2\|_{L^p(w')}
\end{equation}
is the dualised version of the already known weighted estimate of $A_5(b_1, \cdot)$.
Expanding as usual we get that
\begin{align*}
\langle A_5(b_1, f_1)&, b_2f_2\rangle - \langle A_5(b_1, b_2f_1), f_2\rangle \\
&= \sum_{i=1}^2 \langle A_5(b_1, f_1), a_i^1(b_2, f_2) \rangle  \\
&- \sum_{i=1}^8 \langle A_5(b_1, A_i(b_2,f_1)), f_2) \rangle  \\
&+ \sum_{I,J} \lambda_{I,J}^{b_1} \langle f_1, h_I \otimes h_J\rangle \langle (\langle b_2\rangle_{I,1} - \langle b_2 \rangle_{I \times J}) \langle f_2, h_I\rangle_1, 
h_Jh_J\rangle.
\end{align*}
For the last term first estimate
$$
|\langle (\langle b_2\rangle_{I,1} - \langle b_2 \rangle_{I \times J}) \langle f_2, h_I\rangle_1, h_J h_J\rangle| \le \Big\langle \varphi_{\calD^n, b_2}(f_2), h_I \otimes \frac{1_J}{|J|} \Big\rangle,
$$
then use \eqref{eq:A5weighted} and the weighted boundedness of the operator $\varphi_{\calD^n, b_2}$. The first two terms are even
more immediate -- we are done.
\end{proof}

\subsection{Banach range boundedness}
We consider first the Banach range boundedness of $[b_2, [b_1, U]_1]_2$, when $U = U^v_k$
is a general bilinear bi-parameter model operator satisfying \eqref{eq:bb} as in Section \ref{sec:BanachforModels}, and
$\|b_1\|_{\bmo(\R^{n+m})} = \|b_2\|_{\bmo(\R^{n+m})} = 1$. For clarity we pick one explicit $U$:
\begin{equation*}
\begin{split}
\langle U(f_1,f_2),f_3 \rangle
= \sum_{\substack{K \in \calD^n \\ V \in \calD^m}} 
\sum_{\substack{I_1, I_2, I_3 \in \calD^n \\ I_1^{(k_1)} = I_2^{(k_2)} = I_3^{(k_3)} = K}} 
&\sum_{\substack{J_1, J_2, J_3 \in \calD^m \\ J_1^{(v_1)} = J_2^{(v_2)} = J_3^{(v_3)} = V}} a_{K, V, (I_i), (J_j)} \\
&\times \langle f_1, h_{I_1}^0 \otimes  h_{J_1}\rangle \langle f_2, \wt h_{I_2} \otimes \wt h_{J_2}\rangle  
\langle f_3, h_{I_3} \otimes h_{J_3}^0 \rangle.
\end{split}
\end{equation*}
Lemma \ref{lem:case3} gives that
\begin{align*}
&\langle [b_1,U]_1(f_1,f_2),f_3 \rangle \\
&= \sum_{i=1}^2 \sum_{K, V, (I_i), (J_j)} a_{K, \ldots} \langle f_1, h_{I_1}^0 \otimes  h_{J_1}\rangle \langle f_2, \wt h_{I_2} \otimes \wt h_{J_2}\rangle
\langle a_i^1(b_1,f_3), h_{I_3} \otimes h_{J_3}^0 \rangle \\
&+  \sum_{K, V, (I_i), (J_j)} a_{K, \ldots} 
\langle f_1, h_{I_1}^0 \otimes  h_{J_1}\rangle \langle f_2, \wt h_{I_2} \otimes \wt h_{J_2}\rangle
\bla (\langle b_1 \rangle_{I_3,1}-\langle b_1 \rangle_{I_3 \times J_3}) \langle f_3 ,h_{I_3} \rangle_1, h^0_{J_3} \bra \\
&- \sum_{i=1}^2 \sum_{K, V, (I_i), (J_j)} a_{K, \ldots} \langle a_i^2(b_1,f_1), h_{I_1}^0 \otimes  h_{J_1}\rangle \langle f_2, \wt h_{I_2} \otimes \wt h_{J_2}\rangle \langle f_3, h_{I_3} \otimes h_{J_3}^0 \rangle \\
&+ \sum_{K, V, (I_i), (J_j)} a_{K, \ldots} \bla (\langle b_1 \rangle_{I_1 \times J_1} - \langle b_1 \rangle_{J_1, 2})\langle f_1, h_{J_1}\rangle_2, h_{I_1}^0\bra \langle f_2, \wt h_{I_2} \otimes \wt h_{J_2}\rangle \langle f_3, h_{I_3} \otimes h_{J_3}^0 \rangle \\
&+  \sum_{K, V, (I_i), (J_j)} a_{K, \ldots} [\langle b \rangle_{I_3 \times J_3} - \langle b \rangle_{I_1 \times J_1}]
 \langle f_1, h_{I_1}^0 \otimes  h_{J_1}\rangle \langle f_2, \wt h_{I_2} \otimes \wt h_{J_2}\rangle  
\langle f_3, h_{I_3} \otimes h_{J_3}^0 \rangle \\
&=: I + II + III + IV + V.
\end{align*}

When considering $[b_2, [b_1, U]_1]_2$ the first line $I_1$ from above leads to the term
\begin{align*}
\sum_{i=1}^2 &\sum_{K, V, (I_i), (J_j)} a_{K, \ldots} \langle f_1, h_{I_1}^0 \otimes  h_{J_1}\rangle \langle f_2, \wt h_{I_2} \otimes \wt h_{J_2}\rangle
\langle a_i^1(b_1, b_2f_3), h_{I_3} \otimes h_{J_3}^0 \rangle \\
&- \sum_{i=1}^2 \sum_{K, V, (I_i), (J_j)} a_{K, \ldots} \langle f_1, h_{I_1}^0 \otimes  h_{J_1}\rangle \langle b_2f_2, \wt h_{I_2} \otimes \wt h_{J_2}\rangle
\langle a_i^1(b_1,f_3), h_{I_3} \otimes h_{J_3}^0 \rangle.
\end{align*}
We add and subtract
$$
\sum_{i=1}^2 \sum_{K, V, (I_i), (J_j)} a_{K, \ldots} \langle f_1, h_{I_1}^0 \otimes  h_{J_1}\rangle \langle f_2, \wt h_{I_2} \otimes \wt h_{J_2}\rangle
\langle b_2a_i^1(b_1,f_3), h_{I_3} \otimes h_{J_3}^0 \rangle,
$$
so that we need to consider
\begin{align*}
I_{1} := \sum_{i=1}^2 \langle &U(f_1, f_2), a_i^1(b_1, b_2f_3) - b_2a_i^1(b_1,f_3) \rangle \\
&= -\sum_{i=1}^2 \langle U(f_1, f_2), [b_2, a_i^1(b_1, \cdot)](f_3)\rangle
\end{align*}
and
\begin{align*}
I_{2} := \sum_{i=1}^2 \langle [b_2, U]_2(f_1, f_2), a_i^1(b_1,f_3)\rangle.
\end{align*}
Lemma \ref{lem:weightedoneparcommutator} in particular gives
$$
\|[b_2, a_i^1(b_1, \cdot)](f_3)\|_{L^s(\R^{n+m})} \lesssim \|f_3\|_{L^s(\R^{n+m})}, \qquad s \in (1,\infty).
$$
This, together with the boundedness of $U$, takes care of the term $I_{1}$. The term $I_{2}$ is handled using the already
known boundedness of the commutator $[b_2, U]_2$, and the boundedness of $a_i^1(b_1, \cdot)$, $i=1,2$.

When considering $[b_2, [b_1, U]_1]_2$ the term $II$ from above leads to the term
\begin{align*}
&\sum_{K, V, (I_i), (J_j)} a_{K, \ldots} \langle f_1, h_{I_1}^0 \otimes  h_{J_1}\rangle \langle f_2, \wt h_{I_2} \otimes \wt h_{J_2}\rangle
\bla (\langle b_1 \rangle_{I_3,1}-\langle b_1 \rangle_{I_3 \times J_3}) \langle b_2f_3 ,h_{I_3} \rangle_1, h^0_{J_3} \bra \\
&-\sum_{K, V, (I_i), (J_j)} a_{K, \ldots} \langle f_1, h_{I_1}^0 \otimes  h_{J_1}\rangle \langle b_2f_2, \wt h_{I_2} \otimes \wt h_{J_2}\rangle
\bla (\langle b_1 \rangle_{I_3,1}-\langle b_1 \rangle_{I_3 \times J_3}) \langle f_3 ,h_{I_3} \rangle_1, h^0_{J_3} \bra.
\end{align*}
Here we simply start following our original strategy of expanding $b_2f_3$ and $b_2f_2$. How $b_2f_2$ is expanded will, of course, depend
on the the Haar functions $\wt h_{I_2}$ and $\wt h_{J_2}$. So we first expand $b_2f_3$. The first line from above can then be written as the sum of
\begin{align*}
\sum_{i=1}^2 \sum_{K, V, (I_i), (J_j)} a_{K, \ldots} \langle f_1, h_{I_1}^0& \otimes  h_{J_1}\rangle \langle f_2, \wt h_{I_2} \otimes \wt h_{J_2}\rangle \\
&\times \bla (\langle b_1 \rangle_{I_3,1}-\langle b_1 \rangle_{I_3 \times J_3}) \langle a_i^1(b_2,f_3) ,h_{I_3} \rangle_1, h^0_{J_3} \bra,
\end{align*}
\begin{align*}
 \sum_{K, V, (I_i), (J_j)}& a_{K, \ldots} \langle f_1, h_{I_1}^0 \otimes  h_{J_1}\rangle \langle f_2, \wt h_{I_2} \otimes \wt h_{J_2}\rangle \\
 &\times \bla (\langle b_1 \rangle_{I_3,1}-\langle b_1 \rangle_{I_3 \times J_3})(\langle b_2 \rangle_{I_3,1}-\langle b_2 \rangle_{I_3 \times J_3})  \langle f_3,h_{I_3} \rangle_1, h^0_{J_3} \bra
\end{align*}
and
\begin{align*}
\sum_{K, V, (I_i), (J_j)} a_{K, \ldots}\langle b_2\rangle_{I_3 \times J_3} \langle f_1, h_{I_1}^0& \otimes  h_{J_1}\rangle \langle f_2, \wt h_{I_2} \otimes \wt h_{J_2}\rangle \\
&\times \bla (\langle b_1 \rangle_{I_3,1}-\langle b_1 \rangle_{I_3 \times J_3}) \langle f_3 ,h_{I_3} \rangle_1, h^0_{J_3} \bra.
\end{align*}
We take care of the first two terms -- the third term is of course not handled alone. The first term just uses the estimate from Lemma \ref{lem:maximalbound}
saying that
\begin{align*}
\big|\bla (\langle b_1 \rangle_{I_3,1}-\langle& b_1 \rangle_{I_3 \times J_3}) \langle a_i^1(b_2,f_3) ,h_{I_3} \rangle_1, h^0_{J_3} \bra\big| \\
&\le \langle \varphi_{\calD^n, b_1}(a_i^1(b_2,f_3)), h_{I_3} \otimes h_{J_3}^0\rangle,
\end{align*}
then the boundedness \eqref{eq:bb}, and finally the bound
$$
\|  \varphi_{\calD^n, b_1}(a_i^1(b_2,f_3)) \|_{L^{r'}(\R^{n+m})} \lesssim \| a_i^1(b_2,f_3)\|_{L^{r'}(\R^{n+m})}
\lesssim \|f_3\|_{L^{r'}(\R^{n+m})}.
$$
The second term goes in a similar way. Indeed, notice that you can use natural maximal functions, like
$$
M_{\langle b_1 \rangle_{I, 1}, \langle b_2 \rangle_{I, 1}} f = \sup_{J \subset \R^m} \frac{1_J}{|J|} \int_J |\langle b_1 \rangle_{I, 1} - 
\langle b_1 \rangle_{I \times J}||\langle b_2 \rangle_{I, 1} -  \langle b_2 \rangle_{I \times J}| |f|,
$$
for which results as in Lemma \ref{lem:bmaxbounds} hold with essentially the same proof.

We are now ready to expand $b_2f_2$. To avoid a case chase we assume for simplicity that $\wt h_{I_2} = h_{I_2}^0$ and
$\wt h_{J_2} = h_{J_2}$. Then we have
\begin{align*}
\langle b_2f_2, h_{I_2}^0 \otimes h_{J_2}\rangle &= \sum_{i=1}^2 \langle a^2_i(b_2,f_2), h_{I_2}^0 \otimes h_{J_2}\rangle \\
&+ \langle (\langle b_2 \rangle_{J_2,2} - \langle b_2 \rangle_{I_2 \times J_2}) \langle f_2, h_{J_2} \rangle_2, h_{I_2}^0 \rangle
+  \langle b_2 \rangle_{I_2 \times J_2} \langle  f_2, h_{I_2}^0 \otimes h_{J_2}\rangle.
\end{align*}
It is easy to handle
\begin{align*}
-\sum_{i=1}^2 \sum_{K, V, (I_i), (J_j)} a_{K, \ldots} \langle f_1, h_{I_1}^0 \otimes & h_{J_1}\rangle \langle a^2_i(b_2,f_2), h_{I_2}^0 \otimes h_{J_2}\rangle \\
&\times \bla (\langle b_1 \rangle_{I_3,1}-\langle b_1 \rangle_{I_3 \times J_3}) \langle f_3 ,h_{I_3} \rangle_1, h^0_{J_3} \bra
\end{align*}
and
\begin{align*}
- \sum_{K, V, (I_i), (J_j)} a_{K, \ldots} \langle f_1, h_{I_1}^0 \otimes & h_{J_1}\rangle \langle (\langle b_2 \rangle_{J_2,2} - \langle b_2 \rangle_{I_2 \times J_2}) \langle f_2, h_{J_2} \rangle_2, h_{I_2}^0 \rangle \\
&\times \bla (\langle b_1 \rangle_{I_3,1}-\langle b_1 \rangle_{I_3 \times J_3}) \langle f_3 ,h_{I_3} \rangle_1, h^0_{J_3} \bra,
\end{align*}
and so we are only left with
\begin{align*}
\sum_{K, V, (I_i), (J_j)} a_{K, \ldots}[\langle b_2\rangle_{I_3 \times J_3}-\langle b_2& \rangle_{I_2 \times J_2}] \langle f_1, h_{I_1}^0 \otimes  h_{J_1}\rangle \langle f_2, h_{I_2}^0 \otimes  h_{J_2}\rangle \\
&\times \bla (\langle b_1 \rangle_{I_3,1}-\langle b_1 \rangle_{I_3 \times J_3}) \langle f_3 ,h_{I_3} \rangle_1, h^0_{J_3} \bra,
\end{align*}
which is again easy (using Lemma \ref{lem:maximalbound}, Lemma \ref{lem:bmobound} and \eqref{eq:bb}). We are done with the contribution of
$II$ to  $[b_2, [b_1, U]_1]_2$.

The contribution of $III$ to $[b_2, [b_1, U]_1]_2$ simply is
$$
- \sum_{i=1}^2 \langle [b_2,U]_2(a_i^2(b_1,f_1), f_2), f_3\rangle,
$$
which is readily in control. The contributions of the terms $IV$ and $V$ to $[b_2, [b_1, U]_1]_2$ are similarly easy, but they require
running our usual argument, instead of getting an easy formula like in the case $III$. We have now taken care of $[b_2, [b_1, U]_1]_2$.

With some thought we can see that above type arguments also take care of a commutator of the form $[b_2, [b_1, U]_1]_1$.
We have proved the following theorem.
\begin{thm}\label{thm:com2ofmodelBanach}
Let $p,q,r \in (1,\infty)$, $1/p + 1/q = 1/r$, $0 \le k_i \in \Z$ and $0 \le v_i \in \Z$, $i=1,2,3$. 

Let $U = U^v_k$ be a general bilinear bi-parameter model operator satisfying
\eqref{eq:bb}. In particular, $U$ can be a bilinear bi-parameter shift, partial paraproduct or full paraproduct.
Then for $b_1, b_2$ such that $\|b_1\|_{\bmo(\R^{n+m})} = \|b_2\|_{\bmo(\R^{n+m})} = 1$ we have
\begin{align*}
\|[b_2, [b_1, U]_1]_2(f_1, f_2)\|_{L^r(\R^{n+m})} + \|[b_2,& [b_1, U]_1]_1(f_1, f_2)\|_{L^r(\R^{n+m})} \\
&\lesssim (1+\max(k_i, v_i))^2 \|f_1\|_{L^p(\R^{n+m})} \|f_2\|_{L^q(\R^{n+m})}.
\end{align*}
It follows that if $T$ is a bilinear bi-parameter singular integral satisfying the assumptions of Theorem \ref{thm:rep} then also
\begin{align*}
\|[b_2, [b_1, T]_1]_2(f_1, f_2)\|_{L^r(\R^{n+m})} + \|[b_2, [b_1, &T]_1]_1(f_1, f_2)\|_{L^r(\R^{n+m})} \\
&\lesssim \|f_1\|_{L^p(\R^{n+m})} \|f_2\|_{L^q(\R^{n+m})}.
\end{align*}
\end{thm}
We content with the formulation of the above theorem, and do not explicitly iterate more.
\subsection{Quasi--Banach estimates}
Our first goal is to prove:
\begin{thm}\label{thm:com2ofmodelQuasiBanach}
Let $\|b_1\|_{\bmo(\R^{n+m})} = \|b_2\|_{\bmo(\R^{n+m})} = 1$, and let $1 < p, q \le \infty$ and $1/2 < r < \infty$ satisfy $1/p+1/q = 1/r$.
Suppose $S_{\omega, \omega'} := S^{v}_{k, \mathcal{D}^n_{\omega},\mathcal{D}^m_{\omega'}}$
is a bilinear bi-parameter shift of complexity $(k,v)$ defined using the dyadic grids $\mathcal{D}^n_{\omega}$ and $\mathcal{D}^m_{\omega'}$.
Then we have
$$
 \|\mathbb{E}_{\omega, \omega'}[b_2,[b_1,S_{\omega, \omega'}]_1]_2(f_1, f_2)\|_{L^r(\R^{n+m})} 
 \lesssim (1+\max(k_i, v_i))^2 \|f_1\|_{L^p(\R^{n+m})} \|f_2\|_{L^q(\R^{n+m})},
$$
and similarly for $\mathbb{E}_{\omega, \omega'}[b_2,[b_1,S_{\omega, \omega'}]_1]_1$.
\end{thm}
As in Section \ref{sec:quasiviarest} we need to prove the following:
Given $p,q \in (1, \infty)$ and $r \in (1/2,1)$ satisfying $1/p+1/q=1/r$, 
$f_1 \in L^p(\R^{n+m})$, $f_2 \in L^q(\R^{n+m})$ and a set $E \subset \R^{n+m}$ with $0 < |E| < \infty$,
there exists a subset $E' \subset E$ such that $|E'| \ge |E|/2$ and such that for all functions $f_3$
satisfying $|f_3| \le 1_{E'}$ there holds
\begin{equation}\label{eq:ResAveShift2}
\begin{split}
| \langle \mathbb{E}_{\omega, \omega'}&[b_2,[b_1,S_{\omega, \omega'}]_1]_2(f_1, f_2),f_3 \rangle | \\
&\lesssim (1+\max(k_i, v_i))^2 \|f_1\|_{L^p(\R^{n+m})} \|f_2\|_{L^q(\R^{n+m})}|E|^{1/r'}.
\end{split}
\end{equation}

For definiteness let $S_{\omega, \omega'}$ be again of the form \eqref{eq:ResAveShiftMixed}:
$$
\langle S_{\omega, \omega'}(f_1, f_2), f_3 \rangle = \sum_{\substack{K \in \calD^n_0 \\ V \in \calD^m_0}}  A^{\omega, \omega'}_{K \times V}(f_1, f_2, f_3) 
$$
where
\begin{equation*}
\begin{split}
A^{\omega, \omega'}_{K \times V}(f_1, f_2, f_3) 
= \sum_{\substack{I_1, I_2, I_3 \in \calD^n_\omega \\ I_i^{(k_i)} = K+\omega}} &
\sum_{\substack{J_1, J_2, J_3 \in \calD^m_{\omega'} \\ J_i^{(v_i)}  = V+\omega'}} 
a^{\omega,\omega'}_{K+\omega, V+\omega', (I_i), (J_j)}  \\
&\times  \langle f_1,  h_{I_1}^0 \otimes   h_{J_1}\rangle \langle f_2,  h_{I_2} \otimes  h_{J_2}\rangle  
\langle f_3, h_{I_3} \otimes  h_{J_3}^0 \rangle.
\end{split}
\end{equation*}

Similarly as in the Banach range case above
we start with the identity given by Lemma \ref{lem:case3}:
\begin{align*}
&\langle [b_1,S_{\omega, \omega'}]_1(f_1,f_2),f_3 \rangle \\
&= \sum_{i=1}^2 \sum_{K, V, (I_i), (J_j)} a_{K+\omega, \ldots}^{\omega, \omega'} \langle f_1, h_{I_1}^0 \otimes  h_{J_1}\rangle \langle f_2, h_{I_2} \otimes h_{J_2}\rangle
\langle a_{i,\omega}^1(b_1,f_3), h_{I_3} \otimes h_{J_3}^0 \rangle \\
&+  \sum_{K, V, (I_i), (J_j)} a_{K+\omega, \ldots}^{\omega, \omega'} 
\langle f_1, h_{I_1}^0 \otimes  h_{J_1}\rangle \langle f_2, h_{I_2} \otimes h_{J_2}\rangle
\bla (\langle b_1 \rangle_{I_3,1}-\langle b_1 \rangle_{I_3 \times J_3}) \langle f_3 ,h_{I_3} \rangle_1, h^0_{J_3} \bra \\
&- \sum_{i=1}^2 \sum_{K, V, (I_i), (J_j)} a_{K+\omega, \ldots}^{\omega, \omega'}\langle a_{i,\omega'}^2(b_1,f_1), h_{I_1}^0 \otimes  h_{J_1}\rangle \langle f_2, h_{I_2} \otimes h_{J_2}\rangle \langle f_3, h_{I_3} \otimes h_{J_3}^0 \rangle \\
&+ \sum_{K, V, (I_i), (J_j)} a_{K+\omega, \ldots}^{\omega, \omega'} \bla (\langle b_1 \rangle_{I_1 \times J_1} - \langle b_1 \rangle_{J_1, 2})\langle f_1, h_{J_1}\rangle_2, h_{I_1}^0\bra \langle f_2, h_{I_2} \otimes h_{J_2}\rangle \langle f_3, h_{I_3} \otimes h_{J_3}^0 \rangle \\
&+  \sum_{K, V, (I_i), (J_j)} a_{K+\omega, \ldots}^{\omega, \omega'} [\langle b \rangle_{I_3 \times J_3} - \langle b \rangle_{I_1 \times J_1}]
 \langle f_1, h_{I_1}^0 \otimes  h_{J_1}\rangle \langle f_2, h_{I_2} \otimes h_{J_2}\rangle  
\langle f_3, h_{I_3} \otimes h_{J_3}^0 \rangle \\
&=: I + II + III + IV + V.
\end{align*}

We fix $i \in \{1,2\}$ and start considering the corresponding term of $I$, namely
$$
\sum_{\substack{K \in \calD^n_0 \\ V \in \calD^m_0}}  A^{\omega, \omega'}_{K \times V}(f_1, f_2, a_{i,\omega}^1(b_1, f_3)),
$$
and its contribution to $[b_2, [b_1, U]_1]_2$. This leads to the need to prove the following lemma.
\begin{lem}\label{lem:ResAveShiftEx2}
Let $\|b_1\|_{\bmo(\R^{n+m})} = \|b_2\|_{\bmo(\R^{n+m})} = 1$ and let $p,q \in (1, \infty)$ and $r \in (1/2,1)$ satisfy $1/p+1/q=1/r$. 
Suppose $f_1 \in L^p(\R^{n+m})$, $f_2 \in L^q(\R^{n+m})$ and  $E \subset \R^{n+m}$ with $0 < |E| < \infty$.
Then there exists a subset $E' \subset E$ with $|E'| \ge \frac{99}{100} |E|$  so that for all functions $f_3$
satisfying $|f_3| \le 1_{E'}$ there holds
\begin{equation}\label{eq:ResAveShiftExEst2}
\begin{split}
\Big|\E_{\omega,\omega'}\sum_{\substack{K \in \calD^n_0 \\ V \in \calD^m_0}} &
[A^{\omega, \omega'}_{K \times V}(f_1, f_2, a_{i,\omega}^1(b_1, b_2f_3)) 
- A^{\omega, \omega'}_{K \times V}(f_1, b_2f_2, a_{i,\omega}^1(b_1, f_3))] 
 \Big| \\
 &\lesssim (1+\max(k_i,v_i))\|f_1\|_{L^p(\R^{n+m})} \|f_2\|_{L^q(\R^{n+m})}|E|^{1/r'}.
\end{split}
\end{equation}
\end{lem}
\begin{proof}
We assume
\begin{align*}
\Big\| S^2_{v_1}f_1 S^{k_2,v_2}f_2 + \sum_{j=1}^8  S^2_{v_1}f_1 S^{k_2,v_2}_{A_{j,b_2}} f_2 \Big\|_{L^r} = 1, 
\end{align*}
where $A_{j,b_2}$ denotes the family $\{A_{j,\omega,\omega'}(b_2, \cdot)\}_{\omega,\omega'}$, and
the square functions are defined as in Lemma \ref{lem:DetSquareShift}. Define
$$
\Omega_u =\Big\{S^2_{v_1}f_1 S^{k_2,v_2}f_2 +  \sum_{j=1}^8 S^2_{v_1}f_1 S^{k_2,v_2}_{A_{j,b_2}} f_2
> C_0 2^{-u}|E|^{-1/r}\Big\}, \quad u \ge 0,
$$
and 
$$
\wt \Omega_u = \{M1_{\Omega_u}> c_1\},
$$
where $c_1>0$ is a small enough dimensional constant. Then we can
choose $C_0=C_0(c_1)$ so large that the set $E':=E \setminus \wt \Omega_0$ satisfies $|E'|  \ge \frac{99}{100} |E|$.
Then we define the collections
$$
\widehat \calR_u =\Big\{ R \in \calD_0 \colon |R \cap \Omega_u | \ge \frac{|R|}{2}\Big\},
$$
where $\calD_0 = \calD^n_0 \times \calD^m_0$, and set $\calR_u = \widehat\calR_u \setminus \widehat \calR_{u-1}$ for $u \ge 1$. 
Fix now some function $f_3$ such that $|f_3| \le 1_{E'}$.
Lets abbreviate
$$
\Lambda_{K \times V}^{\omega, \omega'}(f_1,f_2,f_3) = 
A^{\omega, \omega'}_{K \times V}(f_1, f_2, a_{i,\omega}^1(b_1, b_2f_3)) 
- A^{\omega, \omega'}_{K \times V}(f_1, b_2f_2, a_{i,\omega}^1(b_1, f_3)).
$$
Notice  the localisation property
$$
\Lambda_{K \times V}^{\omega, \omega'}(f_1,f_2,f_3)
=\Lambda_{K \times V}^{\omega, \omega'}(f_1,f_2,1_{(K+\omega) \times (V+\omega')} f_3).
$$
Based on this, using an argument as in the proof of Lemma \ref{lem:DetSquareShift}, and splitting 
$\Lambda_{K \times V}^{\omega, \omega'}(f_1,f_2,f_3)$ as in \eqref{eq:Com2Split1} and
\eqref{eq:Com2Split2}, we see that we may write
\begin{align*}
\E_{\omega,\omega'}\sum_{\substack{K \in \calD^n_0 \\ V \in \calD^m_0}} 
 \Lambda^{\omega,\omega'}_{K \times V} (f_1,f_2,f_3 )
=\sum_{u=1}^\infty \sum_{K \times V\in \calR_u} 
\E_{\omega,\omega'} \Lambda^{\omega,\omega'}_{K \times V} (f_1,f_2, 1_{\wt \Omega_u} f_3 ).
\end{align*}

We now fix $u$, and our goal is to prove
$$
\sum_{K \times V\in \calR_u} 
\E_{\omega,\omega'} |\Lambda^{\omega,\omega'}_{K \times V} (f_1,f_2, 1_{\wt \Omega_u} f_3 )| \lesssim
(1+\max(k_i,v_i))2^{-u(1-r)}|E|^{1/r'}.
$$
By adding and subtracting
\begin{align*}
A^{\omega, \omega'}_{K \times V}(f_1, f_2, b_2a_{i,\omega}^1(b_1, 1_{\wt \Omega_u}f_3)),
\end{align*}
we see that
\begin{equation}\label{eq:Com2Split1}
\begin{split}
\Lambda^{\omega,\omega'}_{K \times V} (f_1,f_2, 1_{\wt \Omega_u} f_3 ) 
&= -A^{\omega, \omega'}_{K \times V}(f_1, f_2, [b_2,a_{i,\omega}^1(b_1, \cdot)](1_{\wt \Omega_u}f_3)) \\
&+ A^{\omega, \omega'}_{K \times V}(f_1, f_2, b_2a_{i,\omega}^1(b_1, 1_{\wt \Omega_u}f_3)) \\
&- A^{\omega, \omega'}_{K \times V}(f_1, b_2f_2, a_{i,\omega}^1(b_1, 1_{\wt \Omega_u}f_3)).
\end{split}
\end{equation}
Using the symmetric form of Lemma \ref{lem:case2} we have that the last difference
$$
A^{\omega, \omega'}_{K \times V}(f_1, f_2, b_2a_{i,\omega}^1(b_1, 1_{\wt \Omega_u}f_3)) - A^{\omega, \omega'}_{K \times V}(f_1, b_2f_2, a_{i,\omega}^1(b_1, 1_{\wt \Omega_u}f_3))
$$
equals 
\begin{equation}\label{eq:Com2Split2}
\begin{split}
\sum_{j=1}^2 &A^{\omega, \omega'}_{K \times V}(f_1, f_2, a_{j,\omega}^1(b_2,a_{i,\omega}^1(b_1, 1_{\wt \Omega_u}f_3))) \\
&- \sum_{j=1}^8 A^{\omega, \omega'}_{K \times V}(f_1, A_{j, \omega, \omega'}(b_2,f_2), a_{i,\omega}^1(b_1, 1_{\wt \Omega_u}f_3)) \\
&+\sum_{\substack{I_1, I_2, I_3 \in \calD^n_\omega \\ I_i^{(k_i)} = K+\omega}} 
\sum_{\substack{J_1, J_2, J_3 \in \calD^m_{\omega'} \\ J_i^{(v_i)}  = V+\omega'}}
\Big[a^{\omega,\omega'}_{K+\omega,\dots} \langle f_1,h^0_{I_1}\otimes h_{J_1} \rangle 
\langle f_2,h_{I_2}\otimes h_{J_2} \rangle  \\
& \hspace{2cm}\times\bla(\langle b_2 \rangle_{I_3,1}-\langle b_2 \rangle_{I_3\times J_3}) 
\langle a_{i,\omega}^1(b_1, 1_{\wt \Omega_u}f_3),h_{I_3} \rangle_1, h_{J_3}^0 \bra \Big]\\
&+\sum_{\substack{I_1, I_2, I_3 \in \calD^n_\omega \\ I_i^{(k_i)} = K+\omega}} 
\sum_{\substack{J_1, J_2, J_3 \in \calD^m_{\omega'} \\ J_i^{(v_i)}  = V+\omega'}}
\Big[a^{\omega,\omega'}_{K+\omega,\dots}
(\langle b_2 \rangle_{I_3 \times J_3}-\langle b_2 \rangle_{I_2 \times J_2})
 \langle f_1,h^0_{I_1}\otimes h_{J_1} \rangle \\
& \hspace{2cm} \times \langle f_2,h_{I_2}\otimes h_{J_2} \rangle
\langle a_{i,\omega}^1(b_1, 1_{\wt \Omega_u}f_3),h_{I_3} \otimes h_{J_3}^0 \rangle \Big].
\end{split}
\end{equation}
We consider the first term of \eqref{eq:Com2Split1} and the terms from  \eqref{eq:Com2Split2} separately.
These are all handled quite similarly, the only difference being what square functions appear.

Let's  first take a look at the term from \eqref{eq:Com2Split1}. Denote the family of operators $\{ [b_2, a^1_{i,\omega}(b_1, \cdot)]\}_\omega$ 
by $[b_2,a^1_{i,b_1}]$.
Estimating similarly as in 
\eqref{eq:SquaresAppear}, we see that
\begin{equation*}
\begin{split}
\sum_{K\times V \in \calR_u}  & \E_{\omega,\omega'} 
\big| A^{\omega, \omega'}_{K \times V}(f_1, f_2, [b_2,a_{i,\omega}^1(b_1, \cdot)](1_{\wt \Omega_u}f_3)) \big| \\
& \lesssim \int_{\wt \Omega_u \setminus \Omega_{u-1}} S^{2}_{v_1}f_1 S^{k_2,v_2} f_2 S^1_{k_1,[b_2,a^1_{i,b_1}]}(1_{\wt \Omega_u}f_3).
\end{split}
\end{equation*}
From here the estimate can be concluded in the familiar way, using that 
$$
S^{2}_{v_1}f_1 S^{k_2,v_2} f_2 \lesssim 2^{-u}|E|^{-1/r}
$$
in the complement of $\Omega_{u-1}$ and that the square function $S^1_{k_1,[b_2,a^1_{i,b_1}]}$ is bounded. The boundedness
of this square function follows from Lemma \ref{lem:weightedoneparcommutator} and Lemma \ref{lem:DetSquareShift}.

Similarly,  consider for instance the terms from the second line of \eqref{eq:Com2Split2}. For fixed $j=1,\dots,8$ we have
\begin{equation*}
\begin{split}
 \sum_{K\times V \in \calR_u}  & \E_{\omega,\omega'} 
\big| A^{\omega, \omega'}_{K \times V}(f_1, A_{j, \omega, \omega'}(b_2,f_2), a_{i,\omega}^1(b_1, 1_{\wt \Omega_u}f_3))\big| \\
& \lesssim \int_{\wt \Omega_u \setminus \Omega_{u-1}} S^{2}_{v_1}f_1 S^{k_2,v_2}_{A_{j,b_2}} f_2 S^1_{k_3,a^1_{i,b_1}}(1_{\wt \Omega_u}f_3),
\end{split}
\end{equation*}
and the rest is finished as usual.

As a final example we consider  the terms from the third line of \eqref{eq:Com2Split2}. We have using Lemma \ref{lem:maximalbound} that
\begin{equation*}
\begin{split}
\big|\bla(\langle b_2 \rangle_{I_3,1}-\langle b_2 \rangle_{I_3\times J_3}) 
& \langle a_{i,\omega}^1(b_1, 1_{\wt \Omega_u}f_3),h_{I_3} \rangle_1, h_{J_3}^0 \bra \big| \\
&\lesssim \langle \varphi_{\omega, b_2} (a_{i,\omega}^1(b_1, 1_{\wt \Omega_u}f_3)), h_{I_3}\otimes h_{J_3}^0  \rangle,
\end{split}
\end{equation*}
where $\varphi_{\omega, b_2} := \varphi_{\calD^n_{\omega}, b_2}$.
Therefore, there holds that
\begin{equation*}
\begin{split}
  \sum_{K\times V \in \calR_u}  & \E_{\omega,\omega'}
\Big|\sum_{\substack{I_1, I_2, I_3 \in \calD^n_\omega \\ I_i^{(k_i)} = K+\omega}} 
\sum_{\substack{J_1, J_2, J_3 \in \calD^m_{\omega'} \\ J_i^{(v_i)}  = V+\omega'}}
\Big[a^{\omega,\omega'}_{K+\omega,\dots} \langle f_1,h^0_{I_1}\otimes h_{J_1} \rangle 
\langle f_2,h_{I_2}\otimes h_{J_2} \rangle  \\
& \hspace{2cm}\times\bla(\langle b_2 \rangle_{I_3,1}-\langle b_2 \rangle_{I_3\times J_3}) 
\langle a_{i,\omega}^1(b_1, 1_{\wt \Omega_u}f_3),h_{I_3} \rangle_1, h_{J_3}^0 \bra \Big] \Big| \\
& \lesssim \int_{\wt \Omega_u \setminus \Omega_{u-1}} S^{2}_{v_1}f_1 S^{k_2,v_2} f_2 S^1_{k_3,\varphi^1_{b_2} a^1_{i,b_1}}(1_{\wt \Omega_u}f_3),
\end{split}
\end{equation*}
where we wrote $\varphi^1_{b_2} a^1_{i,b_1}$ to mean the family $\{\varphi_{\omega, b_2} a^1_{i,\omega}(b_1,\cdot)\}_\omega$.
Once again, one can finish as before.

The remaining terms, that is the first and fourth ones from \eqref{eq:Com2Split2} go in the same way. Notice that the fourth term produces the factor 
$(1+\max(k_i,v_i))$ to the final estimate.

\end{proof}
The contributions of the terms $II, \ldots, IV$ to the commutator $[b_2, [b_1, U]_1]_2$ can also be handled, and they are easier (the localisation property is more readily available).
Moreover, the fact that we assumed $S_{\omega, \omega'}$ to be of the form \eqref{eq:ResAveShiftMixed} did not play a big role: the other
forms can be handled similarly. Therefore, we get \eqref{eq:ResAveShift2}. With the Banach range boundedness and the usual interpolation this gives us
$$
 \|\mathbb{E}_{\omega, \omega'}[b_2,[b_1,S_{\omega, \omega'}]_1]_2(f_1, f_2)\|_{L^r(\R^{n+m})} 
 \lesssim (1+\max(k_i, v_i))^2 \|f_1\|_{L^p(\R^{n+m})} \|f_2\|_{L^q(\R^{n+m})}
$$
whenever $1 < p, q < \infty$ and $1/2 < r < \infty$ satisfy $1/p+1/q = 1/r$. We can similarly prove
$$
 \|\mathbb{E}_{\omega, \omega'}[b_2,[b_1,S_{\omega, \omega'}]_1]_1(f_1, f_2)\|_{L^r(\R^{n+m})} 
 \lesssim (1+\max(k_i, v_i))^2 \|f_1\|_{L^p(\R^{n+m})} \|f_2\|_{L^q(\R^{n+m})}
$$
in the same range. Then e.g. the case $p=\infty$ for $[b_2,[b_1,S_{\omega, \omega'}]_1]_2$ follows by using the formula
$$
[b_2,[b_1,S_{\omega, \omega'}]_1]_2^{1*} = [b_2, [b_1, S_{\omega, \omega'}^{1*}]_1]_1 -  [b_2, [b_1, S_{\omega, \omega'}^{1*}]_1]_2.
$$
We have proved Theorem \ref{thm:com2ofmodelQuasiBanach}. We end the paper by stating the corresponding corollary for paraproduct free singular integrals $T$.

\begin{thm}\label{thm:com2ofmodelQuasiBanach}
Let $\|b_1\|_{\bmo(\R^{n+m})} = \|b_2\|_{\bmo(\R^{n+m})} = 1$, and let $1 < p, q \le \infty$ and $1/2 < r < \infty$ satisfy $1/p+1/q = 1/r$.
Let $T$ be a bilinear bi-parameter singular integral satisfying the assumptions of Theorem \ref{thm:rep} and assume also
that $T$ is free of paraproducts.
Then we have
$$
 \|[b_2,[b_1,T]_1]_2(f_1, f_2)\|_{L^r(\R^{n+m})} 
 \lesssim \|f_1\|_{L^p(\R^{n+m})} \|f_2\|_{L^q(\R^{n+m})},
$$
and similarly for $[b_2,[b_1,T]_1]_1$.
\end{thm}

\end{document}